\documentclass[12pt,twoside,a4paper]{article}
\usepackage{epstopdf}
\usepackage{amsmath}
\usepackage{amsthm}
\usepackage{amssymb}
\usepackage{amsfonts}
\usepackage{mathrsfs}
\usepackage[all]{xy}
\usepackage{color}
\usepackage{pdfcolmk}
\usepackage{mathtools}
\usepackage{comment}

\numberwithin{equation}{subsection}

\theoremstyle{plain}
\newtheorem{theorem}{Theorem}

\newtheorem{corollary}{Corollary}

\newtheorem{lemma}{Lemma}

\newtheorem{proposition}{Proposition}

\theoremstyle{definition}
\newtheorem{definition}{Definition}

\theoremstyle{remark}
\newtheorem{example}{Example}

\newtheorem{remark}{Remark}

\DeclareFontFamily{OT1}{wncyi}{}
\DeclareFontShape{OT1}{wncyi}{m}{it}{
<5> <6> <7> <8> <9> gen * wncyi
<10> <10.95> <12> <14.4> <17.28> <20.74> <24.88> wncyi10
}{}
\DeclareSymbolFont{cyrletters}{OT1}{wncyi}{m}{it}
\DeclareSymbolFontAlphabet{\cyrmath}{cyrletters}
\DeclareMathSymbol{\rE}{\cyrmath}{cyrletters}{003}
\DeclareMathSymbol{\rD}{\cyrmath}{cyrletters}{068}
\DeclareMathSymbol{\rG}{\cyrmath}{cyrletters}{017}
\DeclareMathSymbol{\rI}{\cyrmath}{cyrletters}{073}
\DeclareMathSymbol{\rL}{\cyrmath}{cyrletters}{076}
\DeclareMathSymbol{\rZ}{\cyrmath}{cyrletters}{090}

\renewcommand{\phi}{\varphi}

\newcommand{\Mc}{\mathcal{M}}

\newcommand{\Rbf}{\mathbf{R}}

\newcommand{\Zbf}{\mathbf{Z}}

\newcommand{\Qbf}{\mathbf{Q}}

\newcommand{\F}{\mathbb{F}}

\newcommand{\Cc}{\mathcal{C}}

\newcommand{\Ac}{\mathcal{A}}

\newcommand{\Fb}{\mathbb{F}}
\newcommand{\Db}{\mathbb{D}}
\newcommand{\A}{\mathbb{A}}

\newcommand{\C}{\mathbb{C}}
\newcommand{\R}{\mathbb{R}}

\newcommand{\Q}{{\mathbb{Q}}}

\newcommand{\Z}{\mathbb{Z}}

\newcommand{\N}{\mathbb{N}}

\newcommand{\Pb}{\mathbb{P}}

\newcommand{\NSets}{\textsc{NSets}}

\newcommand{\Lip}{\mathrm{Lip}}

\newcommand{\Spec}{\mathrm{Spec}}

\newcommand{\Multilin}{\mathrm{Multilin}}

\newcommand{\Mod}{\textsc{Mod}}
\newcommand{\SMod}{\textsc{SMod}}

\newcommand{\Monads}{\textsc{Monads}}

\newcommand{\CBanhalos}{\textsc{CBanhalos}}

\newcommand{\Sym}{\mathrm{Sym}}

\newcommand{\Hom}{\mathrm{Hom}}

\newcommand{\Map}{\mathrm{Map}}

\newcommand{\lf}{\mathrm{leaves}}
\newcommand{\col}{\mathrm{col}}

\newcommand{\uEnd}{\underline{\mathrm{End}}}
\newcommand{\uHom}{\underline{\mathrm{Hom}}}

\newcommand{\Frac}{\mathrm{Frac}}

\newcommand{\Grp}{\textsc{Grp}}

\newcommand{\Alg}{\textsc{Alg}}
\newcommand{\CAlg}{\textsc{CAlg}}

\newcommand{\End}{\mathrm{End}}

\newcommand{\GL}{\mathrm{GL}}

\newcommand{\Iso}{\mathrm{Iso}}

\newcommand{\SIso}{\mathrm{SIso}}

\newcommand{\Res}{\mathrm{Res}}

\title{Banach halos\\
and\\
short isometries}

\author{Tomoki Mihara\footnote{mihara@math.tsukuba.ac.jp}\\
Fr\'ed\'eric Paugam\footnote{frederic.paugam@imj-prg.fr}}

\begin{document}

%\begin{frontmatter}

\maketitle

\begin{abstract}
The aim of this article is twofold.
First, we develop the notion of a Banach halo, similar to that of a Banach ring, except
that the usual triangular inequality is replaced by the inequality
$$|a+b|_R\leq \|(|a|_R,|b|_R)\|_p$$
involving the $\ell^p$-norm for some $p\in ]0,+\infty]$, or
by the inequality
$$|a+b|_R\leq C\max(|a|_R,|b|_R)$$
for $C\in \R_+^*$.
This allows us to put on Banach halos a natural $\R_+^*$-flow,  and to propose a setting for
normed algebra over bases such as $(\Z,|\cdot|_\infty^2)$ where $|\cdot|_\infty^2$ is the square
of the usual archimedean norm.
Then we define and study the groups
$\SIso(C,\sigma,\|\cdot\|)$ of short isometries of normed involutive coalgebras over a base commutative Banach halo $R$.
An aim of this theory is to define a representable
group $K_n\subset \GL_n$ over $\Z$ whose points with values in $\R$
give $O_n(\R)$, and whose points with values in $\Q_p$ give $\GL_n(\Z_p)$, giving
to the classical analogy between these two groups a kind of geometric explanation.
%Moreover, the points of the group $K_n$ with values in $\Z$ give the group $O_n(\R)\cap \GL_n(\Z)$.
%may be thought of
%as points of the classical algebraic group $\GL_n$
%with values in the hypothetic field with one element $\Fb_{\{\pm 1\}}$.
\footnote{
{\bf Keywords:} Algebras with involution, Compact analytic groups, Global analytic geometry.}
\footnote{{\bf MSC classification (2000):} 18D10, 14G22, 14G25, 11G35, 18C15.}
\end{abstract}

\newpage
\tableofcontents
\newpage

\section{Introduction}
A classical leitmotiv of analytic number theory, following Tate's thesis \cite{Tate-these},
is to treat all places of a number field on equal footing, using adelic
methods. Another possible way to interpret this idea, is to try to develop
normed algebra and global analytic geometry (see \cite{Berkovich1}, \cite{Poineau})
in a setting where one can define a natural $\R_+^*$-flow on Banach
rings by setting
$$\sigma_t(R,|\cdot|_R):=(R,|\cdot|_R^t).$$
We propose here a table of analogies between geometry over $\C$, over
a finite field $\F_p$ and arithmetic geometry ($D$ denotes a disc), to show the importance of this flow in analytic geometry:
$$
\begin{array}{|c|c|c|c|}
\hline
\textrm{Geometry} & \textrm{Over }\C & \textrm{Over }\F_p &  \textrm{Arithmetic} \\
\hline
\textrm{Ring} & \C[z] & \F_p[t]  & \Z \\
\hline
\textrm{Origin} & (z)  & (t)  & (p) \\
\hline
\textrm{Unit ``disc''} & D(z=0,1)_{\C} \coloneqq & D(t=0,1)_{\F_p} \coloneqq & D(p=0,1)_{\Z} \coloneqq \\
\textrm{centered at the origin} &  \{z \in \C \mid |z| \leq 1\} & \Mc(\F_p[t],|\cdot|_{triv})  & \Mc(\Z,|\cdot|_{triv}) \\
\hline
\textrm{``Disc'' of radius } p & D(z=0,p)_{\C} \coloneqq  & D(t=0,p)_{\F_p}  & D(p=0,p)_{\Z} \\
\textrm{centered at the origin} & \{z \in \C \mid |z| \leq p\} & \Mc(\F_p[t],|\cdot|_{t,p})  & \Mc(\Z,|\cdot|_{\infty}) \\
\hline
\textrm{Affine space} & \A^{1,an}_\C \coloneqq \Mc(\C[z]) = \C = & \A^{1,an}_{\F_p} \coloneqq \Mc(\F_p[t]) = & \Mc(\Z) \coloneqq \\
& %\bigcup_{s \in [0,\infty)} D(z=0,p)_{\C}^s 
& \bigcup_{s \in [0,\infty)} D(t=0,p)_{\F_p}^s & \bigcup_{s \in [0,\infty)} D(p=0,p)_{\Z}^s \\
\hline
\textrm{Compactification} & \Pb^{1,an}_\C \coloneqq \C \Pb^1 = & \Pb^{1,an}_{\F_p} \coloneqq & \overline{\Mc(\Z)} \coloneqq \\
& \A^{1,an}_\C \cup \{\infty\} & \A^{1,an}_{\F_p} \cup \{\infty\} & \Mc(\Z) \cup \{\infty\} ? \\
\hline
\textrm{Symmetries} & \GL_{n,\C} & \GL_{n,\Fb_p} & \GL_n(\Z)\cap O_n(\R)\\
\hline

\end{array}
$$

The usual triangular inequality
$$|a+b|_R\leq |a|_R+|b|_R$$
is not stable by this flow on multiplicative semi-norms
because (for example) the square $|\cdot|_\infty^2$ of
the archimedean absolute value on $\Z$ doesn't fulfill it.
A possible solution to this problem, proposed by Artin in \cite{E-Artin1}
is to use the generalized triangular inequality
$$|a+b|_R\leq C\cdot \max(|a|_R,|b|_R)$$
for some $C\geq 1$. Developing normed algebra directly in this setting
give quite different notions of what one usually gets when the
usual triangular inequality is fulfilled.
We propose here to also use the family
of $\ell^p$-norm triangular inequalities
$$|a+b|_R\leq \|(|a|_R,|b|_R)\|_p,$$
for varying $p\in ]0,+\infty]$ because it is also indeed
stable by the above $\R_+^*$-flow.

We thus define analogs of Banach rings in those two settings,
called respectively short and Lipschitz Banach halos,
and develop in Chapter 2 the
basics of normed algebra over Banach halos. An interesting property of this setting is
that there is a natural embedding of Banach halos in the category of
monads on normed sets, that is compatible with the natural $\R_+^*$-flows
on both categories. In Chapter 3, we study the groups of short involutions
of Banach coalgebras with involutions and show that under some natural finiteness
hypothesis, they are representable. We finish by applying this to the definition
and study of the subgroups of short isometries of various classical groups
over $(\Z,|\cdot|_\infty)$.

%************************************************************************
\section{Banach halos and their modules}
%************************************************************************
\subsection{Normed sets and semi-normed abelian groups}
We denote $\R_+=[0,+\infty[$ and $\R_+^*=]0,+\infty[$.
Recall from \cite{Fred-overconvergent-global-analytic-geometry} the following definition and results.
\begin{definition}
A normed set is a pair $(X,|\cdot|_X)$ composed of a set $X$ and a map $|\cdot|_X:X\to \R_+$.
A map $f:(X,|\cdot|_X)\to (Y,|\cdot|_Y)$ between two normed sets is called
\begin{enumerate}
\item a Lipschitz map if there exists $C\geq 0$ with $|f(x)|_Y\leq C\cdot |x|_X$ for every $x\in X$.
\item a short map if $|f(x)|_Y\leq |x|_X$ for every $x\in X$, i.e., if it is Lipschitz with constant $C=1$.
\end{enumerate}
We denote $\NSets_{\mathrm{short}}$ the category of normed sets with short maps,
and $\NSets_{\mathrm{Lipschitz}}$ the category of normed sets and Lipschitz maps.
\end{definition}

The category $\NSets_{\mathrm{short}}$ is complete and cocomplete, but the category $\NSets_{\mathrm{Lipschitz}}$ is not.
However, it will be useful for some purposes.

\begin{proposition}
There is a natural multiplicative symmetric monoidal structure on normed sets given by setting
$(X,|\cdot|_X)\otimes^m (Y,|\cdot|_Y)$ to be $X\times Y$ with the norm given by
$$
(x,y)_{X\otimes^m Y}:=|x|_X\cdot |y|_Y.
$$
The internal homomorphism object for this monoidal structure on $\NSets_{\mathrm{short}}$ is given by the set
$$
\uHom(X,Y):=\{f:X\to Y,\;\exists C\geq 0,\;|f(x)|_Y\leq C|x|_X\textrm{ for all }x\in X\}
$$
of Lipschitz maps with the operator norm given by
$$\|f\|:=\inf\{C\geq 0,\;;|f(x)|_Y\leq C|x|_X\textrm{ for all }x\in X\}.$$
For each $p\in ]0,+\infty]$, there is a symmetric monoidal structure $\oplus^p$ on normed sets given by
setting $(X,|\cdot|_X)\oplus^p (Y,|\cdot|_Y)$ to be $X\times Y$ with the norm given by
$$
(x,y)_{X\oplus^p Y}:=\|(|x|_X,|y|_Y)\|_p=\sqrt[p]{|x|_X^p+|y|_Y^p}
$$
if $p<+\infty$ and
$$
(x,y)_{X\oplus^\infty Y}:=\|(|x|_X,|y|_Y)\|_\infty=\max(|x|_X,|y|_Y).
$$
\end{proposition}

The symmetric monoidal structures give variations of norms of the direct product of copies of a normed set. Since we will use them later, we prepare the convention here.

\begin{definition}
\label{abstract sum}
Let $X$ be a normed set with the underlying set $X_0$, $n \in \N \setminus \{0\}$, and $p \in [1,\infty]$. We denote by $\|\cdot\|_{X,p} \colon X_0^n \to [0,\infty[$ the norm of $X^{\oplus^p n}$.
\end{definition}

There is a natural $\R_+^*$-flow on normed sets given by
$$\sigma_t(X,|\cdot|_X)=(X,|\cdot|_X^t)$$
and it acts on both categories $\NSets_{\mathrm{short}}$ and $\NSets_{\mathrm{Lipschitz}}$ and stabilizes $\otimes^m$ and $\oplus^\infty$,
but modifies $\oplus^p$.

Since $\ell^p$ norms will play an important role in this text,
we recall here the following classical result, for the reader's convenience.
\begin{lemma}
\label{p-norm-q-norm}
Let $0<q\leq p\leq \infty$. There is a natural short inclusion
$$
\ell^q(\N,\R_+)\subset \ell^p(\N,\R_+).
$$
\end{lemma}
\begin{proof}
Suppose given a sequence $x = (x_k)_{k \in \N} \in \ell^q(\N,\R_+)$. We want to prove the inequality
$$\|x\|_{p}\leq \|x\|_{q}.$$

If $x=0$, it is clear. Suppose $x\neq 0$.
Suppose first that $p<\infty$. We have
$$\frac{|x_k|}{\|x\|_{p}}\leq 1$$
so that
$$\frac{|x_k|^p}{\|x\|_{p}^p}\leq \frac{|x_k|^q}{\|x\|_{p}^q},$$
for any $k \in \N$, and summing over $k$ gives
$$1=\frac{\|x\|_{p}^p}{\|x\|_{p}^p}\leq \frac{\|x\|_{q}^q}{\|x\|_{p}^q}.$$
This implies
$$\|x\|_{p}^q\leq \|x\|_{q}^q$$
and the desired inequality follows by taking $q$-th roots.
If $p=\infty$, the inequality is also clear.
\end{proof}

\begin{definition}
A short semi-norm on an abelian group $A$ is a norm $|\cdot|_A:A\to \R_+$ on the underlying set of $A$
such that $|0|_A=0$ and there exists $p\in ]0,+\infty]$ such that the $p$-triangular inequality
$$|a-b|_A\leq \|(|a|_A,|b|_A)\|_{p}$$
is fulfilled for every $(a,b)\in A^2$.
A Lipschitz semi-norm on an abelian group $A$ is a norm $|\cdot|_A:A\to \R_+$ such that $|0|_A=0$ and
there exists $C\in \R_+^*$ such that the substraction map is $C$-Lipschitz, i.e.,
$$|a-b|_A\leq C\cdot\max(|a|,|b|).$$
\end{definition}

A Lipschitz semi-normed abelian group is thus an abelian group object in the cartesian monoidal category
$(\NSets_{\mathrm{Lipschitz}},\oplus^\infty)$, and a short semi-normed abelian group is an abelian group object in one
of the monoidal categories $(\NSets_{\mathrm{short}},\oplus^p)$.

\begin{remark}
If $|\cdot|_A$ is a semi-norm on an abelian group $A$,
there is a maximal $p\in ]0,+\infty]$ such that
$$|a-b|_A\leq \|(|a|_A,|b|_A)\|_{p}$$
for every $(a,b)\in A^2$. We will denote it $p^{max}_A$.
\end{remark}

To every semi-normed abelian group $(A,|\cdot|_A)$ one may associate the corresponding completion
$(\hat{A},|\cdot|_{\hat{A}})$,
defined in the usual way as the quotient of Cauchy sequences by sequences with limit $0$, equipped
with the quotient norm, defined by
$$|\overline{(a_n)}|_{\hat{A}}:=\lim_{n\to \infty} |a_n|_A.$$
A semi-normed abelian group is called a Banach abelian group if the natural morphism
$$(A,|\cdot|_A)\to (\hat{A},|\cdot|_{\hat{A}})$$
is both a (short or Lipschitz) isomorphism of normed sets and of abelian groups. This is equivalent to
having the fact that every Cauchy sequence in $(A,|\cdot|_A)$ converges.

%************************************************************************
\subsection{Banach halos}
A Banach halo is a ring that is a Banach abelian group with a sub-multiplicative semi-norm.
More precisely, a short Banach halo is a ring object in the bimonoidal category $(\NSets_{\mathrm{short}},\otimes^m,\oplus^p)$
for some $p\in ]0,+\infty]$, and a Lipschitz Banach halo is a ring object in the bimonoidal category
$(\NSets_{\mathrm{Lipschitz}},\otimes^m,\oplus^\infty)$. Let us write all this explicitely.
\begin{definition}
A short Banach halo is a triad $(R,|\cdot|_R,p_R)$ composed of a ring $R$, a map $|\cdot|_R \colon R \to \R_+$,
and a constant $p_R \in ]0,+\infty]$ satisfying the following:
\begin{enumerate}
\item $|1|_R = 1$ unless $R=\{0\}$.
\item For any $f \in R, |f|_R = 0$ if and only if $f = 0$.
\item For any $f,g \in R, |f-g|_R \leq \|(|f|_R,|g|_R)\|_{p_R}$.
\item For any $f,g \in R, |fg|_R \leq |f|_R\cdot |g|_R$.
\item The semi-normed abelian group $(R,|\cdot|_R)$ is complete.
\end{enumerate}
\end{definition}

For a $p = (C,D) \in ]0,\infty]^2$, we denote by $\|\cdot\|_p$ the seminorm on $\R^2$ defined by
$$\|(x,y)\|_p:=C\cdot \max(|x|,|y|)$$
for $(x,y) \in \R^2$.
Although the convention redundantly refers to the unused datum $D$, using the same convention as in the short setting allows us to make parallel arguments. We note that we do not introduce a notation like $\|\cdot\|_{X,p}$ for a normed set $X$, because it leads confusion due to the fact that the $\ell^{\infty}$-norm rescaled by $C$ does not give an appropriate norm for a direct sum of modules with suitable norms introduced in \S \ref{Direct sums}.

\begin{definition}
A Lipschitz Banach halo is a triad $(R,|\cdot|_R,p_R)$ composed of a ring $R$, a map
$|\cdot|_R\colon R\to \R_+$, and a pair $p_R = (C_R,D_R)$ of two constants $C_R>0$ and $D_R>0$ satisfying the following:
\begin{enumerate}
\item $|1|_R = 1$ unless $R=\{0\}$.
\item For any $f \in R, |f|_R = 0$ if and only if $f = 0$.
\item For any $f,g \in R, |f-g|_R \leq \|(|f|_R,|g|_R)\|_{p_R}$.
\item For any $f,g \in R, |fg|_R \leq D_R\cdot |f|_R\cdot  |g|_R$.
\item The semi-normed abelian group $(R,|\cdot|_R)$ is complete.
\end{enumerate}
\end{definition}

We equip $[0,\infty]^2$ with the direct product order of the opposite order of the usual order. Namely, $p_1 \leq p_2$ holds for a $p_1 = (C_1,D_1) \in [0,\infty]^2$ and a $p_2 = (C_2,D_2) \in [0,\infty]^2$ if and only if $C_2 \leq C_1$ and $D_2 \leq D_1$.

\vspace{0.1in}
Let $\Sigma$ denote either ``short'' or ``Lipschitz''. A $\Sigma$ morphism $(R_1,|\cdot|_{R_1},p_{R_1}) \to (R_2,|\cdot|_{R_2},p_{R_2})$ between $\Sigma$ Banach halos is a $\Sigma$ ring homomorphism $(R_1,|\cdot|_{R_1}) \to (R_2,|\cdot|_{R_2})$ with $p_{R_1} \leq p_{R_2}$.

For a short Banach halo $(R,|\cdot|_R,p_R)$, we denote by ${\rm Lip}(R,|\cdot|_R,p_R)$ the triad $(R,|\cdot|_R,(2^{\frac{1}{p_R}},1))$, which forms a Lipschitz Banach halo because
$$\|(r,s)\|_{p_R}\leq (2\cdot \max(r^{p_R},s^{p_R}))^{\frac{1}{p_R}} = 2^{\frac{1}{p_R}}\cdot \max(r,s) = \|(r,s)\|_{(2^{\frac{1}{p_R}},1)}$$
for any $(r,s) \in [0,\infty[^2$. Then ${\rm Lip}$ is a functor from the category of short Banach halos and short morphisms to the category of Lipschitz Banach halos and Lipschitz morphisms.

\begin{definition}
A $\Sigma$ Banach halo is called commutative if the underlying ring is commutative. An algebra over a $\Sigma$ Banach halo
$R$ is a $\Sigma$ morphism $R\to A$ between $\Sigma$ Banach halos.
\end{definition}

One motivation for introducing Banach halos is that they give a replacement of the notion of Banach ring
that is stable by the natural $\R_+^*$-flow on normed sets. Indeed, one may define such a flow on the category of
short Banach halos by the formula
$$
\sigma_t(A):=(A,|\cdot|_A^t,p_A/t)
$$
and on the category of Lipschitz Banach halos by the formula
$$
\sigma_t(A):=(A,|\cdot|_A^t,(C_A^t,D_A^t)).
$$
This flow is already well understood in the non-archimedean context, and the notion of Banach halo allows
its extension to the archimedean context, because for example
$$\sigma_2(\Z,|\cdot|_\infty,1)=(\Z,|\cdot|_\infty^2,1/2)$$
is not a Banach ring but is a Banach halo.

The Lipschitz case was already known to Artin, and the short case allows an easier treatment
of convergent series.

\begin{example}
We now give various examples of short Banach halos.
\begin{enumerate}
\item Let $(R,|\cdot|_R)$ be a Banach ring in the usual sense. Then $(R,|\cdot|_R,1)$ is a short Banach halo.
\item For example, suppose $n>1$ and equip $V=\Z^n$ with the $\ell^2$ norm on $(\Z,|\cdot|_\infty)$.
Then the algebra $\uEnd(V)$ of $\Z$-linear endomorphisms of $\Z$ with its operator norm and with
constant $1$ is a non-commutative short Banach halo (see the proof of Proposition \ref{free-module-p}).
\item Let $(R,|\cdot|_R,p_R)$ be a short Banach halo. Then there is a maximal $p\geq p_R$ such that
$(R,|\cdot|_R,p)$ is a short Banach halo, denoted $p_R^{max}$ and there is a natural morphism
(defined using Lemma \ref{p-norm-q-norm})
$$(R,|\cdot|_R,p_R)\to (R,|\cdot|_R,p_R^{max}).$$
A short Banach halo of this form is called a maximal halo.
\item For example, the commutative short Banach
halos $(\Z,|\cdot|_\infty,1)$, $(\Z,|\cdot|_0,\infty)$, $(\Z_p,|\cdot|_p,\infty)$
and $(\Z,|\cdot|_\infty^2,1/2)$ are all maximal. 
\item There is no initial commutative short Banach halo.
Indeed, the underlying ring would be $\Z$ and the norm on it
should be the limit of all norms $|\cdot|_\infty^t$ for $t$ going to $\infty$, that is not well defined. This problem
may be solved by working with pro-Banach halos (or ind-Banach halos).
\end{enumerate}
\end{example}

%************************************************************************
\subsection{Banach modules}
%and monads on normed sets}
In this Subsection, $\Sigma$ denotes either of the words ``short'' or ``Lipschitz''.
%************************************************************************
\subsubsection{Definitions}
Modules over Banach halos are essentially modules in the bimonoidal categories $(\NSets_{\mathrm{short}},\otimes^m,\oplus^p)$
and $(\NSets_{\mathrm{Lipschitz}},\otimes^m,\oplus^\infty)$ with bounded constants.

\begin{definition}
A module over a short Banach halo $(R,|\cdot|_R,p_R)$ is a triad $(M,|\cdot|_M,p_M)$ of an $R$-module $M$, a map $|\cdot|_M \colon M \to \R_+$ and a constant $p_M\in ]0,+\infty]$ such that $p_M\geq p_R$ and satisfying the following:
\begin{enumerate}
\item For any $m \in M$, $|m|_M = 0$ if and only if $m = 0$.
\item For any $m,n \in M$, $|m-n|_M \leq \|(|m|_M,|n|_M)\|_{p_M}$
\item For any $f \in R$ and $m \in M, |fm|_M \leq |f|_R\cdot |m|_M$
\item The semi-normed abelian group $(M,|\cdot|_M)$ is complete.
\end{enumerate}
A morphisms $f:M_1\to M_2$
between modules $M_1$ and $M_2$ over $R$ is a short $R$-linear map with
$p_{M_1}\leq p_{M_2}$. We denote $\Mod_R$ the category whose objects are
$R$-modules and whose morphisms are short morphisms.
\end{definition}

\begin{definition}
A module over a Lipschitz Banach halo $(R,|\cdot|_R,p_R)$ with $p_R = (C_R,D_R)$ is a triad $(M,|\cdot|_M,p_M)$
of an $R$-module $M$, a map $|\cdot|_M\colon M\to \R_+$ and
a pair $p_M = (C_M,D_M) \in ]0,\infty]^2$ of two constants $C_M$ and $D_M$ such that $p_M \geq p_R$ and satisfying the following:
\begin{enumerate}
\item  For any $m \in M$, $|m|_M = 0$ if and only if $m = 0$.
\item For any $m,n \in M$, $|m-n|_M \leq \|(|m|_M,|n|_M)\|_{p_M}$
\item For any $f \in R$ and $m \in M, |fm|_M \leq D_M\cdot |f|_R\cdot |m|_M$
\item The semi-normed abelian group $(M,|\cdot|_M)$ is complete.
\end{enumerate}
A Lipschitz (resp.\ short) morphisms $f:M_1\to M_2$
between Lipschitz modules $M_1$ and $M_2$ over $R$ is a Lipschitz (resp.\ short) $R$-linear map with $p_{M_1}\leq p_{M_2}$. We denote $\Mod_R$ the category whose objects are Lischitz $R$-modules and whose
morphisms are Lipschitz morphisms and $\SMod_R$ the category with the same
objects and short morphisms.
\end{definition}

For a module $(M,|\cdot|_M,p_M)$ over a short Banach halo $(R,|\cdot|_R,p_R)$, we denote by ${\rm Lip}(M,|\cdot|_M,p_M)$ the triad $(M,|\cdot|_M,(2^{\frac{1}{p_M}},1))$, which forms a Lipschitz module over ${\rm Lip}(R,|\cdot|_R,p_R)$ because
\begin{eqnarray*}
|m-n|_M & \leq & \|(|m|_M,|n|_M)\|_{p_M} \leq (2 \max(|m|_M^{p_M},|n|_M^{p_M}))^{1/p_M} \\
& = & 2^{\frac{1}{p_M}} \max(|m|_M,|n|_M) = \|(|m|_M,|n|_M)\|_{(2^{\frac{1}{p_M}},1)}.
\end{eqnarray*}
Then ${\rm Lip}$ is a functor from the category of short modules over $(R,|\cdot|_R,p_R)$ and short morphisms to the category of Lipschitz modukes over ${\rm Lip}(R,|\cdot|_R,p_R)$ and Lipschitz morphisms.

\begin{definition}
For $p\geq p_R$, and $R$ a short or a Lipschitz Banach halo,
we denote $\Mod_R^p$ (resp. $\SMod_R^p$) the subcategory of $\Mod_R$ (resp. $\SMod_R$) whose objects
are modules $M$ with $p_M\geq p$.
\end{definition}

%************************************************************************
\subsubsection{Endomorphism rings}

Let $\Sigma$ denote either ``short'' or ``Lipschitz''. Let $R$ be a $\Sigma$ Banach halo with the underlying ring $R_0$, and $M$ an $R$-module. We denote by $\End_{R}(M) \subset \uHom(M,M)$ the $R_0$-algebra of Lipschitz $R_0$-linear endomorphisms of $M$, equip it with the norm $|\cdot|_{\End_{R}(M)}$ the operator norm $|\cdot|_{op}$ in the case where $\Sigma$ is ``short'' and the rescaled operator norm $D_M^{-1} |\cdot|_{op}$ in the case where $\Sigma$ is ``Lipschitz'', and with the constant $p_{\End_{R}(M)} = p_M$.

\begin{proposition}
\label{endomorphism ring}
The triad $(\End_{R}(M),|\cdot|_{\End_{R}(M)},p_{\End_{R}(M)})$ forms an $R$-algebra, and $M$ forms a module over it.
\end{proposition}

\begin{proof}
The only non-trivial condition is the 
triangular inequality. Let $f,g \in \End_{R}(M)$. If $\Sigma$ is ``short'', then for any $m \in M$, we have
\begin{eqnarray*}
 |(f - g)m|_M  	& = & |fm - gm|_M \leq \|(|fm|_M,|gm|_M)\|_{p_M}\\
 			& \leq& \|(|f|_{op} \ |m|_M,|g|_{op} \ |m|_M)\|_{p_M} \\
			& = & \|(|f|_{\End_{R}(M)},|g|_{\End_{R}(M)})\|_{p_M} \ |m|_M \\
			& = & \|(|f|_{\End_{R}(M)},|g|_{\End_{R}(M)})\|_{p_{\End_{R}(M)}} \ |m|_M,
\end{eqnarray*}
and hence $|f-g|_{\End_{R}(M)} = |f-g|_{op} \leq \|(|f|_{\End_{R}(M)},|g|_{\End_{R}(M)})\|_{p_{\End_{R}(M)}}$.
%Here, we used the inequality $|\cdot|_{p_M} \leq |\cdot|_{p_{\End_{R}(M)}}$, which follows from Lemma \ref{p-norm-q-norm}. (i don't understand the meaning of this sentence so i took it out)
If $\Sigma$ is ``Lipschitz'', then for any $m \in M$, we have
\begin{eqnarray*}
 |(f - g)m|_M  	& = & |fm - gm|_M \leq \|(|fm|_M,|gm|_M)\|_{p_M}\\
 			& \leq& \|(|f|_{op} \ |m|_M, |g|_{op} \ |m|_M)\|_{p_M} \\
			& = & D_M \|(|f|_{\End_{R}(M)},|g|_{\End_{R}(M)})\|_{p_M} \ |m|_M \\
			& = & D_M \|(|f|_{\End_{R}(M)},|g|_{\End_{R}(M)})\|_{p_{\End_{R}(M)}} \ |m|_M,
\end{eqnarray*}
and hence $|f-g|_{\End_{R}(M)} = D_M^{-1} |f-g|_{op} \leq \|(|f|_{\End_{R}(M)},|g|_{\End_{R}(M)})\|_{p_{\End_{R}(M)}}$.
\end{proof}

%************************************************************************
\subsubsection{Direct sums}
\label{Direct sums}
Let $R$ be a short Banach halo, and $(M_i)_{i \in I}$ a family of short modules over $R$. For each $i \in I$, we denote by $\underline{M}_i$ the underlying Abelian group of $M_i$. For an $m = (m_i)_{i \in I} \in \bigoplus_{i \in I} \underline{M}_i$ and a $p \in [p_R,\infty]$, we put $\|m\|_{I,p}:= \sqrt[p]{\sum_{i \in I} |m_i|_{M_i}^p}$. We introduce a Lipschitz counterpart of $\|\cdot\|_{I,p}$. For this purpose, we introduce convention for binary trees.

\vspace{0.1in}
Let $S$ be a set.
We denote by $BT_S$ the set of finite binary trees whose leaves are coloured by $S$, i.e.\ the set defined in the following recursive way:
\begin{itemize}
\item[(i)] $S \times \{0\} \subset BT_S$.
\item[(ii)] $BT_S^2 \subset BT_S$.
\end{itemize}
For a $t \in BT_S$, we denote by $\lf(t)$ the set of leaves of $t$, i.e.\ the set defined in the following recursive way:
\begin{itemize}
\item[(i)] If $t = (s,0) \in S \times \{0\}$, then $\lf(t) = \{s\}$.
\item[(ii)] If $t = (t_0,t_1) \in BT_S^2$, then $\lf(t) = \bigcup_{i=0}^{1} \lf(t_i) \times \{i\}$.
\end{itemize}
For a $t \in BT_S$, we denote by $\col(t)$ the colouring of $t$, i.e.\ the map $\lf(t) \to S$ defined in the following recursive way:
\begin{itemize}
\item[(i)] If $t = (s,0) \in S \times \{0\}$, then $\col(t)$ is the inclusion $\lf(t) = \{s\} \hookrightarrow S$,
\item[(ii)] If $t = (t_0,t_1) \in BT_S^2$, then $\col(t)$ assigns $\col(t_i)(x) \in S$ to each $(x,i) \in \lf(t)$.
\end{itemize}
For a map $f \colon S_0 \to S_1$ and $t \in BT_{S_0}$, we denote by $f \circ t \in BT_{S_1}$ the replacement of the colouring of $t$ by the composition of $f$.

\vspace{0.1in}
Let $C \in [0,\infty[$. For a $t \in BT_{[0,\infty[}$, we define $\|t\|_C \in [0,\infty[$ in the following recursive way:
\begin{itemize}
\item[(i)] If $t = (s,0) \in [0,\infty[ \times \{0\}$, then $\|t\|_C = s$,
\item[(ii)] If $t = (t_0,t_1) \in BT_S^2$, then $\|t\|_C = C\cdot \max(\|t_0\|_C,\|t_1\|_C)$.
\end{itemize}
Let $R$ be a Lispchitz Banach halo, and $(M_i)_{i \in I}$ a family of Lipschitz modules over $R$. For each $i \in I$, we denote by $\underline{M}_i$ the underlying Abelian group of $M_i$, and by $X$ the normed set $\bigsqcup_{i \in I} M_i$ regarded as a subset of $\bigoplus_{i \in I} \underline{M}_i$. For an $m \in \bigoplus_{i \in I} \underline{M}_i$ and a $C \in ]0,\infty[$, we put $BT(m) := \{t \in BT_X \mid \sum_{x \in \lf(t)} \col(t)(x) = m\}$, and $\|m\|_{I,C} := \inf_{t \in BT(m)} \||\cdot|_X \circ t\|_C \in [0,\infty[$. For a $p = (C,D) \in ]0,\infty]^2$, we put $\|\cdot\|_{I,p} = \|\cdot\|_{I,C}$.

\begin{example}
If the direct sum is of the form $M_1\oplus M_2$ with $I = \{1,2\}$ and 
$C=\inf(C_{M_1},C_{M_2})$, then the Lipschitz
norm is given by
$$
\begin{array}{l}
\|(0,0)\|_{I,C} = 0\\
\|(m_1,0)\|_{I,C} = |m_1|_{M_1} \\
\|(0,m_2)\|_{I,C} = |m_2|_{M_2} \\
\|(m_1,m_2)\|_{I,C} = \\
\inf 
\left\{
\begin{array}{l}
C\max(|m_1|_{M_1},|m_2|_{M_2}), \\
\inf\{C\max(C\max(|m_{1,1}|_{M_1},|m_2|_{M_2}),|m_{1,2}|_{M_1}) \mid (m_{1,j})_{j=1}^{2} \in X_{1,2}\}, \\
\inf\{C\max(C\max(|m_1|_{M_1},|m_{2,1}|_{M_2}),|m_{2,2}|_{M_2}) \mid (m_{2,j})_{j=1}^{2} \in X_{2,2}\}, \\
\underset{(m_{i,j})_{i,j=1}^{2} \in X_{1,2} \times X_{2,2}}{\inf}\{C\max(C\max(C\max(|m_{1,1}|_{M_1},|m_{2,1}|_{M_2}),|m_{2,2}|_{M_2}),|m_{1,2}|_{M_1})\}, \\
\underset{(m_{i,j})_{i,j=1}^{2} \in X_{1,2} \times X_{2,2}}{\inf}\{C\max(C\max(C\max(|m_{1,1}|_{M_1},|m_{2,1}|_{M_2}),|m_{1,2}|_{M_1}),|m_{2,2}|_{M_2})\}, \\
\underset{(m_{i,j})_{i,j=1}^{2} \in X_{1,2} \times X_{2,2}}{\inf}\{C\max(C\max(|m_{1,1}|_{M_1},|m_{2,1}|_{M_2}),C\max(|m_{1,2}|_{M_1}),|m_{2,2}|_{M_2}))\}, \\
\vdots
\end{array}
\right\}
\end{array}
$$
for an $(m_1,m_2) \in (M_1 \setminus \{0\}) \times (M_2 \setminus \{0\})$, where
\begin{eqnarray*}
X_{i,n} & = & \left\{(m_{i,j})_{j=1}^{n} \in M_i^n \middle| m_i = \sum_{j=1}^{n} m_{i,j}\right\} \\
\end{eqnarray*}
for each $i \in I$ and $n \in \N$.
More explicitely, if $M_1=M_2=\Z$ with $C=2$ and the usual archimedean norm $|\cdot|_\infty$ on $\Z$,
then we get
$$
\|2\oplus 2\|_{I,2}=2\cdot \max(|2|_\infty,|2|_\infty)=4.
$$
Indeed, the expression
$$2 \oplus 2 = (2 \oplus 0) + (0 \oplus 2)$$
is the only expression of $2 \oplus 2$ as a binary sum of
elements in the subset $(\Z \oplus \{0\}) \cup (\{0\} \oplus \Z)$.
Therefore all other sum expression of $2 \oplus 2$ are at least
$3$-ary sums. Then the value associated to those sum expression is
at least $C^2 \times 1 = 4$, because the norm of non-zero elements
of is at least 1. This proves that the norm of $2\oplus 2$ is $4$.
\end{example}

\vspace{0.1in}
Let $R$ be a $\Sigma$ Banach halo, and let $(M_i)_{i \in I}$ a family of $\Sigma$ modules over $R$. For each $i \in I$, we denote by $\underline{M}_i$ the underlying Abelian group of $M_i$. When $\Sigma$ is ``short'', we put $p := \sup_{i \in I} p_{M_i}$. When $\Sigma$ is ``Lipschitz'', we put $p := (\inf_{i \in I} C_{M_i},\inf_{i \in I} D_{M_i})$. We have introduced a seminorm $\|\cdot\|_{I,p}$ on $\bigoplus_{i \in I} \underline{M}_i$, and denote by $\bigoplus_{i \in I} M_i$ its completion with respect to $\|\cdot\|_{I,p}$.

\begin{proposition}
The above constructed sum $\bigoplus_{i\in I} M_i$ satisfies the universality of the colimit of $(M_i)_{i \in I}$ in the category $\SMod_R$ of $\Sigma$ modules over $R$ and short morphisms.
\end{proposition}

 Here, the reader should be careful that we are considering short morphisms even when $\Sigma$ is ``Lipschitz''.
%************************************************************************
\subsubsection{Free modules}
Let $\Sigma$ denote either ``short'' or ``Lipschitz''. For a normed set $X$, we denote by $\ell^{p_R}(X,R)$ the $\Sigma$ module $\bigoplus_{x \in X_+} R |x|_X$, where $R r$ for an $r \in ]0,\infty]$ is the regular $\Sigma$ module over $R$ whose norm is rescaled by $r$. We give another description of $\ell^{p_R}(X,R)$ as a completion of $R^{(X)}$ in order to make its universality clearer.

\vspace{0.1in}
Let $\Sigma$ denote either ``short'' or ``Lipschitz''. When $\Sigma$ is ``short'', for a $p \in ]0,\infty]$, we denote by $|\cdot|_{X,p}$ the seminorm on $R^{(X)}$ given by
$$|\sum a_x\{x\}|_{X,p}:=\|(|a_x|_R\cdot |x|_X)\|_{p},$$
When $\Sigma$ is ``Lipschitz'', for a $p = (C,D) \in ]0,\infty]^2$, we denote by $|\cdot|_{X,p}$ the seminorm on $R^{(X)}$ given by
$$|\sum a_x\{x\}|_{X,p}:=\inf \{r \in [0,\infty[ \mid \forall t \in BT((a_x)_{x \in X}), \||\cdot|_Y \circ t\|_C \leq r\},$$
where $Y$ denotes $\bigsqcup_{x \in X} R |x|_X$ and $R r$ for an $r \in [0,\infty[$ denotes the seminormed Abelian group given as the regular Lipdchitz module over $R$ rescaled by $r$. In both cases, $\ell^{p_R}(X,R)$ is naturally identified with the completion of the free module $R^{(X)}$ with respect to $|\cdot|_{X,p_R}$.

\begin{proposition}
\label{free-module}
Let $R$ be a $\Sigma$ Banach halo. Then the forgetful functor $\Mod_R\to \NSets_{\Sigma}$ has a left adjoint functor given by $X \mapsto \ell^{p_R}(X,R)$.
\end{proposition}
\begin{proof}
Denote for simplicity $p=p_R$.
The norm $|\cdot|_{I,p}$ fulfills the $p$-triangular inequality. Indeed, let $(a,b) \in R^{(X)}$ with $a = \sum_{x \in X} a_x \{x\}$ and $b = \sum_{x \in X} b_x \{x\}$. If $\Sigma$ is ``short'', then we have
\begin{eqnarray*}
|a+b|_{X,p}	& =		& \|(|a_x+b_x|_R\cdot |x|_X)_{x\in X}\|_{p}\\
							& \leq 	& \|(\|(|a_x|_R,|b_x|_R)\|_{p}\cdot |x|_X)_{x\in X}\|_{p}\\
							& =		&\sqrt[p]{\sum_{x\in X} (|a_x|_R^p+|b_x|_R^p)\cdot |x|_X^p}\\
							& =		&\sqrt[p]{(\sum_{x\in X} |a_x|_R^p\cdot|x|_X^p)+(\sum_{x\in X} |b_x|_R^p\cdot |x|_X^p)}\\
							& =		& \|(|a|_{X,p},|b|_{X,p})\|_p.
\end{eqnarray*}
If $\Sigma$ is ``Lipschitz'', then for any $t = (t_0,t_1) \in BT((a_x)_{x \in X}) \times BT((b_x)_{x \in X})$, we have $t \in BT((a_x+b_x)_{x \in X})$ and
\begin{eqnarray*}
& & \|||\cdot|_Y \circ t\|_C \leq C \max(\||\cdot|_Y \circ t_0\|_C,\||\cdot|_Y \circ t_1\|_C) \leq C \max(|a|_{X,p},|b|_{X,p}) \\
& = & \|(|a|_{X,p},|b|_{X,p})\|_p,
\end{eqnarray*}
where $Y$ denotes $\bigsqcup_{x \in X} R |x|_X$ as above. We already know that $X\mapsto R^{(X)}$ is a left adjoint in the algebraic setting.
Suppose given an $R$-module $M$ with $p_M\geq p$ and a $\Sigma$ map $i:X\to M$. Then we may extend
$i$ by linearity to $\tilde{i}:R^{(X)}\to M$, and we have
$$
\begin{array}{ccccccc}
\left|\tilde{i}(\sum_{x\in X} a_x\{x\})\right|_M	& = &\left|\sum_{x\in X} a_xi(x)\right|_M & \leq &\|(|a_x|_R\cdot |i(x)|_M)_{x\in X}\|_{p_M}\\
							& \leq&  \|(|a_x|_R\cdot |i|_{op} \cdot |x|_X)_{x\in X}\|_{p_M} & \leq &|i|_{op} \|(|a_x|_R\cdot |x|_X)_{x\in X}\|_{p}
\end{array}
$$
by Lemma \ref{p-norm-q-norm} because $p\leq p_M$. This shows that $\tilde{i}$ extends to the completion into a $\Sigma$ morphism
$$
\ell^{p}(X,R)\to M,
$$
that is a morphism of modules if we further set $p_{\ell^{p}(X,R)}:=p=p_R$. Since the restriction
of the obtained map to $X$ gives the original map, we have obtained a bijection
$$
\Hom_{\Mod_R}((\ell^{p_R}(X,R),|\cdot|_{p_R},p_R),(M,|\cdot|_M,p_M))\overset{\sim}{\longrightarrow}
\Hom_{\NSets_{\Sigma}}(X,M).
$$
\end{proof}

We will need an improvement of the above result, whose proof is more technical.

First, let us prove a Lemma.
\begin{lemma}
\label{enlarge-p}
Let $R$ be a $\Sigma$ Banach halo. Then if $p\geq p_R$, there is an algebra $R_p$
that is universal among algebras over $R$ with constant greater than $p$. 
\end{lemma}
\begin{proof}
Let $(S,|\cdot|_S,p)$ be a triad of a ring $S$, a map $|\cdot|_S \colon S \to \R_+$, and a $p \in (0,\infty]$ (resp.\ $p = (C,D) \in (0,\infty]^2$) satisfying
the following:
\begin{itemize}
\item[(i)] $|1|_S = 1$ (resp.\ $> 0$) unless $S = \{0\}$
\item[(ii)] $|fg|_S \leq |f|_S \cdot |g|_S$ (resp.$D |f|_S \cdot |g|_S$) for any $f,g \in S$
\end{itemize}
By the conditions, we have $|1|_S \in \{0,1\}$ (resp.\ $|1|_S \in (0,\infty)$). We denote by $|\cdot|_{S,p}$ the map
\begin{eqnarray*}
S & \to & \R_+\\
f & \mapsto & \inf \left\{\|(|f_i|_S)_{i=0}^{n-1}\|_{n,p} \mid n \in \N, (f_i)_{i=0}^{n-1} \in S^n, \sum_{i=0}^{n-1} f_i = f\right\},
\end{eqnarray*}
where $n$ is regarded as a normed set $\{i \in \N \mid i < n\}$ equipped with the norm given as the constant map with value $1$, the sum of the empty sequence in $S$ ($(f_i)_{i=1}^{n}$ in the case $n=0$) is defined as $0 \in S$, and the $\ell^p$-norm 
of the empty sequence in $[0,\infty)$ ($(|f_i|_S)_{i=1}^{n}$ in the case $n=0$) is defined as $0 \in [0,\infty)$.
Then the triad $(S,|\cdot|_{S,p},p)$ satisfies the following properties:
\begin{itemize}
\item[(i)] $|1|_{S,p} \in \{0,1\}$ (resp.\ $(0,\infty)$).
\item[(ii)] $|0|_{S,p} = 0$.
\item[(iii)] For any $f,g \in S$, $|f-g|_{S,p} \leq \|(|f|_{S,p},|g|_{S,p})\|_p$.
\item[(iv)] For any $f,g \in S$, $|fg|_{S,p} \leq |f|_{S,p} |g|_{S,p}$ (resp.\ $D |f|_{S,p} |g|_{S,p}$).
\end{itemize}
The condition (ii) immediately follows from the definitions related to the empty sequence. The condition (iii) follows from the 
definition of $|\cdot|_{S,p}$ using the infimum. The condition (i) follows from the condition
(iv). We show the condition (iv). For simplicity, we only consider the case where $\Sigma$ is ``short'' because the computation is quite similar. By the continuity of the multiplication on $\R$, it 
suffices to show  $|fg|_{S,p} \leq (|f|_{S,p}+\epsilon)(|g|_{S,p}+\epsilon)$ for any $\epsilon \in (0,\infty)$.
By the definition of $|f|_{S,p}$, there exists an $(f_i)_{i=1}^{n_f} \in S^{n_f}$ with $n_f \in \N$, 
$\sum_{i=1}^{n_f} f_i = f$, and $0 \leq \|(|f_i|_S)_{i=1}^{n_f}\|_p \leq (|f|_{S,p}+\epsilon)$. By the definition of $|g|_{S,p}$, there 
exists a $(g_j)_{j=1}^{n_g} \in S^{n_g}$ with $n_g \in \N$, $\sum_{j=1}^{n_g} g_j = g$,
and $0 \leq \|(|g_j|_S)_{j=1}^{n_g}\|_p \leq (|g|_{S,p}+\epsilon)$. By the definition of $|fg|_{S,p}$, we have
\begin{eqnarray*}
& & |fg|_{S,p} = \left| \sum_{i=1}^{n_f} f_i \sum_{j=1}^{n_g} g_j \right|_{S,p} = \left| \sum_{i=1}^{n_f} \sum_{j=1}^{n_g} f_i g_j \right|_{S,p} \\
& \leq & \|((|f_ig_j|_S)_{i=1}^{n_f})_{j=1}^{n_g}\|_p = \left( \sum_{i=1}^{n_f} \sum_{j=1}^{n_g} |f_ig_j|_S^p \right)^{\frac{1}{p}} \\
& \leq & \left( \sum_{i=1}^{n_f} \sum_{j=1}^{n_g} |f_i|_S^p |g_j|_S^p \right)^{\frac{1}{p}} = \left( \sum_{i=1}^{n_f} |f_i|_S^p \sum_{j=1}^{n_g} |g_j|_S^p \right)^{\frac{1}{p}} \\
& = & \left( \sum_{i=1}^{n_f} |f_i|_S^p \right)^{\frac{1}{p}} \left( \sum_{j=1}^{n_g} |g_j|_S^p \right)^{\frac{1}{p}} = \|(|f_i|_S)_{i=1}^{n_f}\|_p \|(|g_j|_S)_{j=1}^{n_g}\|_p \\
& \leq & (|f|_{S,p}+\epsilon)(|g|_{S,p}+\epsilon).
\end{eqnarray*}
Therefore the condition (iv) holds.

We denote by $R_p$ the completion of $R$ with respect to the uniform structure associated to $|\cdot|_{R,p}$, i.e. the uniform 
structure generated by the basis $\{\{(f,g) \in R^2 \mid |f-g|_{R,p} < \epsilon\} \mid \epsilon \in (0,\infty)\}$. Since $|\cdot|_{R,p}$ 
satisfies the conditions (i) -- (iv), its extension $|\cdot|_{R_p}$ to $R_p$ also satisfies the conditions, and
$(R,|\cdot|_R)_p := (R_p,|\cdot|_{R_p},p)$ forms a Banach halo.

In addition, if $(R,|\cdot|_R,q)$ forms a Banach halo for a $q \leq p$, then $(R,|\cdot|_R)_p$ forms an $(R,|\cdot|_R,q)$-algebra, 
which is universal in $(R,|\cdot|_R,q)$-algebras whose constants are greater than or equal to $p$.
\end{proof}

\begin{proposition}
\label{free-module-p}
Let $R$ be a $\Sigma$ Banach halo.
For $p\geq p_R$, the forgetful functor $\SMod_R^p\to \NSets_\Sigma$ has a left adjoint.
\end{proposition}
\begin{proof}
Let $X=(X,|\cdot|_X)$ be a normed set.
We need to construct the free $p$-normed module on $X$.
For this purpose, we will use Lemma \ref{enlarge-p} to enlarge $p_R$.

We recall that $\ell^{p_R}(X,R)$ the $R$-module is naturally identified with the completion of
the free module $R^{(X)}$ on the underlying ring of $R$ with respect to the uniform structure associated to $\|\cdot\|_{X,p_R}$.
For any $p \geq p_R$, we denote by $\Res_R^{R_p}\ell^{p}(X,R_p)$ the restriction of scalar of the $R_p$-module $\ell^{p}(X,R_p)$ by
the canonical homomorphism $(R,|\cdot|_R,p_R) \to (R,|\cdot|_R)_p$. Here, the restriction of scalar works by Lemma \ref{p-norm-q-norm}.

We show that $\Res_R^{R_p}\ell^{p}(-,R_p)$ gives the left adjoint functor to the forgetful functor
$$\SMod_R^p\to \NSets_\Sigma.$$
For this purpose, it suffices to show that for any normed set $X$, any $\Sigma$ $R$-module $(M,|\cdot|_M,p_M)$ with $p_M \geq p$, and 
any short map $\varphi \colon X \to M$, the $R$-linear extension $R^{(\varphi)} \colon R^{(X)} \to M$ of $\varphi$ uniquely 
extends to a short morphism $\Res_R^{R_p}\ell^{p}(X,R_p) \to M$ of $\Sigma$ $R$-modules. Indeed, when $\Sigma$ is ``Lipschitz'', the extension property for a Lipschitz morphism follows from that of a short morphism because the rescaling of a norm on $X$ by a positive real number is an isomorphism in $\NSets_\mathrm{Lipschitz}$.
When $p = p_R$, the assertion follows from the natural 
isomorphisms $R\cong R_p$ and $\Res_R^{R_p}\ell^{p}(X,R_p) \cong \ell^{p_R}(X,R)$. Therefore it suffices to show that the $\Sigma$ $R$-module 
structure of $M$ uniquely extends to a $\Sigma$ $R_p$-module structure, because then the $R_p$-linear extension
$R_p^{(\varphi)} \colon R_p^{(X)} \to M$ of $\varphi$ uniquely extends to $\Res_R^{R_p}\ell^{p}(X,R_p)$.

By Lemma \ref{p-norm-q-norm} and Proposition \ref{endomorphism ring}, $(\End_{R}(M),|\cdot|_{\End_{R}(M)},p)$ forms an $R$-algebra and $M$ forms a module over it. Therefore by the universality of $R_p$, $(\End_{R}(M),|\cdot|_{\End_{R}(M)},p)$ forms an $R_p$-algebra, and $M$ forms an $R_p$-module.
\end{proof}

There is a natural flow on monads on normed sets induced by the flow on normed sets,
and defined (Following Nikolai Durov) by
$$\sigma_t\Sigma(X):=\sigma_t\Sigma(\sigma_{1/t}X).$$
Another motivation underlying the definition of Banach halos is given by the following result.
\begin{corollary}
There is a natural fully faithful functor
$$
\begin{array}{ccc}
\CBanhalos_\Sigma 	& \to		& \Monads(\NSets_\Sigma)\\
A			        &\mapsto	& \Lambda_A
\end{array}
$$
from commutative $\Sigma$ Banach halos to monads on normed sets defined by
$$
\Lambda_A(X):=\ell^{p_A}(X,A).
$$
This functor is compatible with the $\R_+^*$-flow.
\end{corollary}
\begin{proof}
In the short situation, the functoriality in $A$ follows from Lemma \ref{p-norm-q-norm}. Indeed, if $f:R\to S$ is a morphism
of Banach halos and $\sum a_x\{x\}\in \Lambda_R(X)$, then we have
$$
|\sum f(a_x)\{x\}|_{p_S}=\|(|f(a_x)|_S\cdot |x|_X)_{x\in X}\|_{p_S}\leq \|(|a_x|_R\cdot |x|_X)_{x\in X}\|_{p_S}\leq
\|(|a_x|_R\cdot |x|_X)_{x\in X}\|_{p_R}
$$
because $p_R\leq p_S$, so that $f$ induces a short map
$\Lambda_R(X)\to \Lambda_S(X)$.
The same result is also true in the Lipschitz setting.
The full faithfulness of the functor follows from the identification
$$\Lambda_R(\{1_1\})=R.$$
The fact that it is compatible with the flow, in the short setting, follows from $p_{\sigma_t R}=p_R/t$ and from the equalities
$$
\begin{array}{ccl}
|\sum a_x\{x\}|_{p_{\sigma_t R}}	& =	& \|(|a_x|_R\cdot |x|_X)\|_{p_{\sigma_t R}}\\
							& =	& \left(\sum |a_x|_{\sigma_t R}^{p_{\sigma_tR}}\cdot |x|^{p_{\sigma_tR}}\right)^{1/p_{\sigma_tR}}\\
							& =	& \left(\left(\sum (|a_x|_R^t)^{p_R/t}\cdot(|x|^{p_R})^{1/t}\right)^{1/p_R}\right)^{t}\\
							& =	& \|(|a_x|_R\cdot |x|_X^{1/t})\|_{p_R}^t
\end{array}
$$
This compatibility with the flow is clear in the Lipschitz setting because the monoidal structure $\oplus^\infty$ on normed set is compatible with the flow.
\end{proof}

Let $S$ be a Lipschitz Banach halo with the underlying ring $S_0$. It is notable that the notion of the freeness of an $R$-module is strictly stronger than the freeness of the underlying $S_0$-module even if it is of finite rank. For example, $\Lip(\Z,|\cdot|_0,\infty)$ is a $\Lip(\Z,|\cdot|_{\infty},1)$-module which is not free but whose underlying $\Z$-module is a free $\Z$-module of rank $1$. Indeed, there exists no non-trivial Lipschitz $\Lip(\Z,|\cdot|_{\infty},1)$-linear homomorphism $\Lip(\Z,|\cdot|_0,\infty) \to \Lip(\Z,|\cdot|_{\infty},1)$. In order to avoid such a pathlogic setting, we give a criterion for the boundedness of $S_0$-linear homomorphisms using $\|\cdot\|_{S,\infty}$ (cf.\ Definition \ref{abstract sum}).

\begin{proposition}
\label{criterion for boundedness}
Let $M$ and $N$ be $S$-modules. Suppose that the underlying $S_0$-module of $M$ is a free $S_0$-module of rank $n \in \N$ with an $S_0$-linear basis $(e_j)_{j=1}^{n} \in M^n$ and there exists a $C \in [0,\infty[$ such that for any $(s_j)_{j=1}^{n} \in S_0^n$, $\|(s_j)_{j=1}^{n}\|_{S,\infty} \leq C |\sum_{j=1}^{n} s_j e_j|_M$. Then every $S_0$-linear homomorphism $f:M \to N$ is Lipschitz.
\end{proposition}

\begin{proof}
Put $C_E = \max{}_{j=1}^{n} |e_j|_M \in ]0,\infty[$ and $C_f = \max{}_{j=1}^{n} |f(e_j)|_N$. For any $(s_j)_{j=1}^{n} \in S_0^n$, we have
\begin{eqnarray*}
& & \left|f \left( \sum_{j=1}^{n} s_j e_j \right)\right|_N = \left|\sum_{j=1}^{n} s_j f(e_j)\right|_N \leq C_N^{n-1} \left\|(s_j f(e_j))_{j=1}^{n}\right\|_{N,\infty} \\
& = & C_N^{n-1} \left\|(|s_j f(e_j)|_N)_{j=1}^{n}\right\|_{{\infty}} \leq C_N^{n-1} \left\|(D_N |s_j|_S |f(e_j)|_N)_{j=1}^{n}\right\|_{{\infty}} \\
& \leq & C_N^{n-1} D_N C_f \left\|(|s_j|_S)_{j=1}^{n}\right\|_{{\infty}} = C_N^{n-1} D_N C_f \left\|(s_j)_{j=1}^{n}\right\|_{S,\infty} \\
& \leq & C_N^{n-1} D_N C_f C \left| \sum_{j=1}^{n} s_j e_j\right|_M.
\end{eqnarray*}
This implies that $f$ is of operator norm $\leq C_N^{n-1} D_N C_f C$.
\end{proof}

%************************************************************************
\subsection{Direct sums, tensor products and symmetric algebras}

Let $\Sigma$ denote either ``short'' or ``Lipschitz''.

\begin{proposition}
Let $R$ be a $\Sigma$ Banach halo and fix $p\geq p_R$.
Let $(M_i)_{i\in I}$ be a family of modules over $R$ such that $\sup_{i\in I}(p_{M_i})=p$.
Then the coproduct of the family in both $\SMod^p_R$ and $\SMod_R$, denoted $\bigoplus_{i\in I} M_i$ exists.
\end{proposition}
\begin{proof}
One may write the direct sum as a quotient
$$
\bigoplus_{i\in I} M_i:=\Res_{R}^{R_p}\ell^{p}\left(\coprod_{i\in I} M_i,R_p\right)/\sim
$$
of the free $p$-normed module of Proposition \ref{free-module-p} by the closure of the appropriate direct sum relations.
This will fulfill the desired universal property.
\end{proof}

\begin{definition}
Let $(M_i)_{i\in I}$ be a finite ordered family of modules over a $\Sigma$ Banach halo $R$ such that $\sup(p_{M_i})=p$.
Then we define the tensor product $\bigotimes_{i\in I}M_i$ as the quotient
$$
\bigotimes_{i\in I} M_i=\Res_{R}^{R_p}\ell^{p}(\otimes_{m,i\in I} M_i,R_p)/\sim
$$
of the free $p$-module of Proposition \ref{free-module-p} on the multiplicative tensor product of the underlying normed sets
by the appropriate bilinearity relations. The corresponding constant is fixed to be $p$.
\end{definition}

\begin{proposition}
\label{universal-property-tensor-product}
Let $(M_i)_{i\in I}$ be a finite ordered family of modules over a 
$\Sigma$ Banach halo $R$ such that $\sup(p_{M_i})=p$.
Then the tensor product $\otimes_{i\in I} M_i$ represents the functor
$$\Multilin((M_i)_{i\in I},-):\Mod_R^p\to \NSets_\Sigma$$
that sends a $R$-module $M$ with $p_M\geq p$ to the set of $R$-multilinear $\Sigma$ maps $f:\otimes_{m,i\in I} M_i\to M$ equipped with the operator norm.
\end{proposition}
\begin{proof}
This follows from Proposition \ref{free-module} and the classical construction of tensor product of modules on rings.
\end{proof}

\begin{proposition}
The forgetful functor $F:\CAlg_R\to \SMod_R$ from commutative algebras over $R$ to the category of $R$-modules with short morphisms has a left adjoint called the symmetric algebra
functor and denoted $M\mapsto \Sym_R(M)$.
\end{proposition}
\begin{proof}
One may describe explicitely the symmetric algebra by
$$
\Sym_R(M):=\bigoplus_{n\in \N}M^{\otimes n}/I_c
$$
where $I_c$ is the closed two sided ideal generated by commutators $x\otimes y-y\otimes x$ in the free associative
algebra $T_R(M):=\bigoplus_{n\in \N}M^{\otimes n}$ over $R$. This gives an adjoint because it is a composition
of two adjoint functors: the adjoint to the forgetful functor from commutative $R$-algebras to associative
$R$-algebras and the adjoint to the forgetful functor from associative $R$-algebras to $R$-modules.
\end{proof}

%************************************************************************
\subsection{Scalar extensions of modules}
\begin{proposition}
\label{extension-of-scalars}
Let $R\to S$ be a $\Sigma$-Banach halo morphism. Then the restriction of scalar functor $\Mod_S\to \Mod_R$
has an adjoint, denoted $\otimes_R S$, and called the scalar extension along $R\to S$.
\end{proposition}
\begin{proof}
The scalar extension of an $R$-module $M$ by a morphism $R\to S$ is given by the tensor product
$M\otimes_R S$, seen as an $S$-module. Proposition \ref{universal-property-tensor-product} implies
that it is an adjoint of the restriction of scalars.
\end{proof}

\begin{example}
We abbreviate $\|\cdot\|_{(\R,|\cdot|_{\infty},q)}$ to $\|\cdot\|_{\infty,\ell^q}$ and $\|\cdot\|_{(\Q_p,|\cdot|_p,q)}$ to $\|\cdot\|_{p,\ell^q}$ for any $q \in [1,\infty)$. Let us consider the case where $\Sigma$ is ``short'', $R = (\Z,|\cdot|_{\infty},1)$,
$S = (\Z_p,|\cdot|_p,\infty)$, and
$M = (\Z^n,|\cdot|_{\infty,\ell^q},1)$ for a $q \in [1,\infty)$.
We have a natural $S$-linear isomorphism $M \otimes^{alg}_{R} S\cong \Z_p^n$, and hence we identify them.
Then $|\cdot|_{M \otimes_{R} S}$ is bounded by $|\cdot|_{p,\ell^{\infty}}$ by the definition of the left hand side
using the infimum. On the other hand, $|\cdot|_{p,\ell^{\infty}}$ is known to satisfy the axiom of a module
over $S$. Therefore for any $m \in M \otimes^{alg}_{R} S$ with a
presentation $\sum_{i=0}^{n} m_i \otimes f_i = m$, we have

$$|m|_{p,\ell^{\infty}} \leq \|(|m_i \otimes 1|_{p,\ell^{\infty}},|f_i|_p)_{i=0}^{n}\|_{{\infty}}
\leq \|(|m_i|_{\infty,\ell^1},|f_i|_p)_{i=0}^{n}\|_{{\infty}}
= \|(|m_i|_M |f_i|_{S})_{i=0}^{n}\|_{{\rho_{S}}}$$

Taking the infimum, we obtain $|m|_{p,\ell^{\infty}} \leq |m|_{M \otimes_{R} S}$.
Thus we obtain $|\cdot|_{p,\ell^{\infty}} = |\cdot|_{M \otimes_{R} S}$. Since $\Z_p^n$ is complete
with respect to $|\cdot|_{p,\ell^{\infty}}$, its completion is isomorphic to itself, so that we get
$$(\Z^n,|\cdot|_{\infty,\ell^q},1)\otimes_{(\Z,|\cdot|_{\infty},1)} (\Z_p,|\cdot|_p,\infty)\cong (\Z_p^n,|\cdot|_{p,\ell^\infty},\infty).$$
\end{example}

We can also apply the completely same argument to other norms $|\cdot|'$ on $\Z^n$, as long as it satisfies the following
two conditions:
\begin{enumerate}
\item The norm of a vector whose entries are 0 except for a single entry 1 is 1. (The condition is used to show
that $|\cdot|_{M \otimes_{R} S}$ is bounded by $|\cdot|_{p,\ell^{\infty}}$.)
\item $(\Z^n,|\cdot|') \to (\Z_p^n,|\cdot|_{p,\ell^{\infty}})$ is short. (This condition is used to
show $\|(|m_i \otimes 1|_{p,\ell^{\infty}},|f_i|_p)_{i=0}^{n}\|_{{\infty}} \leq \|(|m_i|_{\infty,\ell^1},|f_i|_p)_{i=0}^{n}\|_{{\infty}}$ 
above).
\end{enumerate}

\begin{example}
\label{scalar-extension-operator-norm}
In particular, the scalar extension of $M_n(\Z) \cong \Z^{n^2}$ equipped with the $\ell^2$-operator norm along
$(\Z,|\cdot|_\infty,1)\to (\Z_p,|\cdot|_p,\infty)$ is isomorphic to $M_n(\Z_p)$
equipped with the $\ell^{\infty}$ norm. Indeed, the first condition is clear, and the second follows from the following argument.
Let $A \in M_n$. We denote by $A_{i,j}$ the $(i,j)$-entry of $A$ for each $i,j \in \N$ with $1 \leq i,j \leq n$.
Then we have $A \delta_j =(A_{i,j})_{i=1}^{n}$ and hence
$|A|_{\ell^2,op} = |A|_{\ell^2,op} |\delta_j|_{\infty,\ell^2} \geq |(A_{i,j})_{i=1}^{n}|_{\infty,\ell^2} \geq |(A_{i,j})_{i=1}^{n}|_{\infty,\ell^{\infty}} $
(where $|\cdot|_{\infty,\ell^q}$ denotes $\|\cdot\|_{q}$) for
any $j \in \N$ with $1 \leq j \leq n$. Therefore we obtain $|A|_{p,\ell^{\infty}} \leq |A|_{\infty,\ell^{\infty}} \leq |A|_{\ell^2,op}$
under the natural identification of $M_n(\Z)$ and $\Z^{n^2}$.
\end{example}

%************************************************************************
\subsection{Finite colimits of algebras and converging power series}
Another good motivation for working with halos is given by the following result.
\begin{proposition}
Let $A$ be a commutative $\Sigma$ Banach halo. Then finite coproducts exist in the category of commutative $A$-algebras.
\end{proposition}
\begin{proof}
Let $B$ and $C$ be two $A$-algebras.
The $(B,C)$-bimodule structure on $B \otimes C$ is given by Proposition \ref{extension-of-scalars}  applied to the pair
of $B$ and the underlying $A$-module structure of $C$, and the pair of $C$ and the underlying $A$-module structure
of $B$, because of the symmetry of $\otimes$. It gives an $A$-bilinear map $B \otimes_m C \to \End_A(B \otimes C)$,
which extends to $B \otimes C$ as an $A$-module homomorphism by Proposition \ref{universal-property-tensor-product}.
Thus $B \otimes C$ is equipped with a multiplication.
The ring axiom immediately follows from the continuity of the operations and the ring axiom of the image of algebraic
tensor product.
\end{proof}

\begin{corollary}
Let $A$ be a commutative $\Sigma$ Banach halo.
Finite colimits exist in the category $\CBanhalos_A$ of $A$-algebras.
\end{corollary}
\begin{proof}
The coequalizer of a pair $f,g:R\to S$ of parallel arrows is defined as the quotient of $S$
by the closure of the ideal generated by $\{f(x)-g(x),\; x\in R\}$. It fulfills the universal property
of the coequalizer. Since finite coproducts exist in the category of $A$-algebras,
we get the desired result.
\end{proof}

\begin{proposition}
Let $A$ be a commutative $\Sigma$ Banach halo.
The adjoint to the forgetful functor $\Alg_A\to \NSets_\Sigma$ exists and is denoted $X\mapsto A\langle X\rangle$.
The $A$-algebra $A\langle X\rangle$ is called the $A$-algebra of converging power series on the normed
set $X$.
\end{proposition}
\begin{proof}
Let $X$ be a normed set and $(X)^\N_m$ be the corresponding free multiplicatively normed monoid.
Then
$$A\langle X\rangle:=\ell^{p_A}((X)^\N_m,A)$$
is an $A$-algebra that fulfills the desired universal property.
One may also use the composition of the symmetric algebra functor
and of the free module construction to describe the same algebra.
\end{proof}

%************************************************************************
\section{Short isometries and classical compact groups}

%************************************************************************
\subsection{Short isometries of Lipschitz coalgebras}
\label{Short isometries of Lipschitz coalgebras}
Let $R$ be a Lipschitz Banach halo. 
In this section, we abbreviate the tensor product $\otimes_R$ of Lipschitz $R$-modules, which naturally makes sense by extending the construction for modules over a short Banach halo using free modules, to $\otimes$ as long as there is no ambiguity.

\begin{definition}
A monoid (resp.\ comonoid) with involution $(C,\sigma)$ over $R$ is a pair composed 
of a monoid (resp.\ comonoid) object in the
monoidal category of Lipschitz $R$-modules and of an
involution $\sigma$ 
that is a Lipschitz monoid (resp.\ comonoid) morphism $C \to C^{op}$.
\end{definition}

For an $R$-module $M$, we denote by $\uHom_R(M,R)$ the $R$-linear dual of $M$, i.e.\ the $R$-module given as the subset of $\uHom(M,R)$ consisting of Lipschitz $R$-linear homomorphisms equipped with the operator norm $|\cdot|_{\uHom(M,R)} = |\cdot|_{op}$ and the constant $p_{\uHom_R(M,R)} = p_M$. We abbreviate $\uHom_R(M,R)$ to $M^{\vee}$ unless there is no ambiguity of $R$. Then we obtain a contravariant functor $\vee \colon \Mod(R) \to \Mod(R), \ M \mapsto M^{\vee}$.

\begin{proposition}
\label{dual of comonoid}
Let $(C,\sigma)$ be a comonoid with involution over $R$. Then $(C^\vee,\sigma^{\vee})$
is a monoid with involution over $R$.
\end{proposition}
\begin{proof}
For every $R$-module $M$, there is a natural evaluation homomorphism
$$M\otimes M^\vee\to R.$$
This gives an $R$-linear morphism
$$C^\vee\otimes C^\vee\otimes C\otimes C\to R.$$
This induces an $R$-linear morphism
$$C^\vee\otimes C^\vee\to (C\otimes C)^\vee.$$
Composing it with the dual of the comultiplication on $C$, we get a multiplication
$$C^\vee\otimes C^\vee\to C^\vee,$$
for which $C^\vee$ forms a monoid object for $\otimes$. This is equivalent to an $R$-algebra
structure on $A$.
\end{proof}

\begin{remark}
\label{dual of monoid}
We note that the dual statement of Proposition \ref{dual of comonoid} is not necessarily true. In other words, the dual of a monoid with involution $(A,\sigma)$ might not necessarily a comonoid with involution, as the dual of the multiplication $A \otimes A \to A$ is a morphism $A^{\vee} \to (A \otimes A)^{\vee}$ which might not factor through the natural morphism $A^{\vee} \otimes A^{\vee} \to (A \otimes A)^{\vee}$.
\end{remark}

Let $(C,\sigma)$ be a comonoid with involution over $R$.
For a commutative $R$-algebra $S$ and $(c,a) \in C \times (C^{\vee} \otimes S)$, we abbreviate to $a(c) \in S$ the image of $c \otimes a \in C \otimes C^{\vee} \otimes S$ by the scalar extension by $S$ of the pairing $C \otimes C^{\vee} \to R$ associated to the $R$-bilinear map $C \otimes^m C^{\vee} \to R, (c,f) \mapsto f(c)$ by the universality of $\otimes$.

\begin{definition}
We denote by $\CAlg_R$ be the category of commutative $R$-algebras
and bounded $R$-algebra homomorphisms, and define the functor of
short isometries of $(C,\sigma)$, which will be denoted by $\SIso(C,\sigma)$, by 
$$
\begin{array}{rcl}
\CAlg_R   & \to       & \Grp\\
                 S         &\mapsto    &
                \left\{a \in C^{\vee} \otimes S\:
                \begin{array}{|l}
                \exists D \in ]0,\infty[, \forall n \in \N,
                \forall (c_i)_{i=0}^{n-1} \in C^n,\\
                |a(c_0) \cdots a(c_{n-1})|_S \leq D |c_0|_C \cdots |c_{n-1}|_C, \\
                a \sigma_S(a) = \sigma_S(a) a = 1
                \end{array}\right\}
\end{array}
$$
\end{definition}

For a comonoid $C$ in $\Mod_R$ and an $R$-algebra $S$, we denote by $\alpha_{C,S}$ the Lipshictz $S$-linear homomorphism
$$C^{\vee}\otimes_R S = \uHom_R(C,R) \otimes_R S \to \uHom_S(C\otimes_R S,S)$$
assigning to each $a \in C^{\vee} \otimes_R S$ the Lipschitz $S$-linear homomorphism given as the $S$-linear extension of the evaluation map $C \to S, \ c \mapsto a(c)$.

\begin{proposition}
\label{representability-short-isometries-comonoid}
Let $(C,\sigma)$ be a comonoid with involution over $R$ and suppose that for every $R$-algebra $S$, $\alpha_{C,S}$ is an isomorphism.
Then the group of short isometries $\SIso(C,\sigma)$ is represented by the algebra
$$
\Ac(\SIso(C,\sigma)):=\Sym_R(C)/\overline{(\{\nabla((1 \otimes \sigma)(\Delta(c)))-\eta(c) \mid c \in C\})}.
$$
\end{proposition}
\begin{proof}
The comonoid structure $(\Delta,\eta)$ (resp. the involution $\sigma$) on $C$ induces a comultiplication
and a counit (resp. an involution $\sigma$) on the tensor $R$-algebra $(T_R(C),\nabla)$ given by
$$T_R(C)=\bigoplus_{n\in N}C^{\otimes n}$$
by the functoriality of $\oplus$ and $\otimes$, the universality of $\oplus$, and the symmetry of $\otimes$.
Moreover, $T_R(C)$ satisfies the axioms of a bimonoid expect for the coassociativity with respect to the
comultiplication by construction. Those operations induce operations on $\Sym_R(C)$ and $\Ac(\SIso(C,\sigma))$, for which $\Sym_R(C)$ forms a bimonoid and $\Ac(\SIso(C,\sigma))$ forms a Hopf monoid.

Let $S$ be a commutative $R$-algebra. Then every $R$-algebra homomorphism
$\Ac(\SIso(C,\sigma))\to S$ corresponds to an $R$-linear homomorphism $f:C\to S$ such that
$$
\nabla((f\otimes f)((1\otimes \sigma)(\Delta(c))))=\eta(c)
$$
for any $c\in C$. We denote $f_S:C\otimes S\to S$ the scalar extension of $f$ to $S$.
By hypothesis, $f_S$ can be represented as $\alpha_{C,S}(a)$ for a unique $a\in C^\vee\otimes_R S$.
By definition, the image of $c\otimes a$ and $c\otimes \sigma_S^\vee(a)$ by
the natural pairing
$$C\otimes C^\vee\otimes S\to R\otimes S\cong S$$
coincides with $f(c)$ and $f(\sigma(c))$ respectively for any $c\in C$. We now show that
$a\in \SIso(C,\sigma)(S)$.

The converging infinite sums may be permuted in a given halo.
Take a presentation $\sum_{h\in H} a_h \otimes t_h$ of $a$. Let $c \in C \otimes S$. Take a presentation $\sum_{i\in I} c_i \otimes s_i$ of $c$. For each $i\in I$, take a presentation $\sum_{j\in J_i} c_{i,j,1} \otimes c_{i,j,2}$ of $\Delta(c)$. We have
\begin{eqnarray*}
& & \iota_S(a \sigma^{\vee}_S(a) - 1)(c) = \iota_S(a \sigma^{\vee}_S(a) - 1) \left( \sum_{i\in I} c_i \otimes s_i \right) \\
& = & \sum_{h_1,h_2\in H} \sum_{i\in I} \iota_S(a_{h_1} \sigma^{\vee}(a_{h_2}) \otimes t_{h_1} t_{h_2} - 1)(c \otimes s_i) \\
& = & \sum_{h_1,h_2\in H} \sum_{i\in I} \sum_{j\in J_i} \iota_S(a_{h_1} \sigma^{\vee}(a_{h_2}) \otimes t_{h_1} t_{h_2})(c_{i,j,1} \otimes c_{i,j,2} \otimes s_i)\\
& & - \sum_{i\in I} \iota_S(\eta \otimes 1)(c \otimes s_i) \\
& = & \sum_{h_1,h_2\in H} \sum_{i\in I} \sum_{j\in J_i} a_{h_1}(c_{i,j,1}) \sigma^{\vee}(a_{h_2})(c_{i,j,2}) t_{h_1} t_{h_2}s_i
- \sum_{i\in I} \eta(c) s_i \\
& = & \sum_{h_1,h_2\in H} \sum_{i\in I} \sum_{j\in J_i} a_{h_1}(c_{i,j,1}) a_{h_2}(\sigma(c_{i,j,2})) t_{h_1} t_{h_2}s_i
- \sum_{i\in I} \eta(c) s_i\\
& = & \sum_{i\in I} s_i \sum_{j\in J_i} \left( \sum_{h_1\in H} a_{h_1}(c_{i,j,1}) t_{h_1} \right) \left( \sum_{h_2\in H} a_{h_2}(\sigma(c_{i,j,2})) t_{h_2} \right)
- \sum_{i\in I} \eta(c) s_i \\
& = & \sum_{i\in I} s_i \sum_{j\in J_i} \iota_S(a)(c_{i,j,1} \otimes 1) \iota_S(a)(\sigma(c_{i,j,2}) \otimes 1) - \sum_{i\in I} \eta(c) s_i \\
& = & \sum_{i\in I} s_i \sum_{j=1}^{m_i} f_S(c_{i,j,1} \otimes 1) f_S(\sigma(c_{i,j,2}) \otimes 1) - \sum_{i\in I} \eta(c) s_i \\
& = & \sum_{i\in I} s_i \sum_{j\in J_i} f(c_{i,j,1}) f(\sigma(c_{i,j,2})) - \sum_{i\in I} \eta(c) s_i \\
& = & \sum_{i\in I} s_i \nabla((1 \otimes \sigma(\Delta(c))) - \sum_{i\in I} \eta(c) s_i \\
& = & \sum_{i\in I} s_i \eta(c) - \sum_{i\in I} \eta(c) s_i \\
& = & 0,
\end{eqnarray*}
and hence $\iota_S(a \sigma^{\vee}_S(a) - 1) = 0$. Since $\iota_S$ is injective, we obtain $a \sigma^{\vee}_S(a) - 1 = 0$, i.e.\ $a \in \SIso(C,\sigma)$. The opposite implication follows from the same computation.
\end{proof}

A Lipschitz Banach halo $S$ is said to be {\it with submultiplicative norm} if $|f g|_S \leq |f|_S |g|_S$ for any $(f,g) \in S^2$, and is said to be {\it uniform} if $|f^n|_S = |f|_S^n$ for any $(f,n) \in S \times \N$. For example, if $D_R \leq 1$, then every $R$-algebra $S$ is with submultiplicative norm by $D_S \leq D_R = 1$.

\begin{proposition}
\label{inclusion of SIso}
For any commutative $R$-algebra $S$ with submultiplicative norm,
$$\{a \in C^{\vee} \otimes S \mid \forall c \in C, |a(c)|_S \leq |c|_C, a \sigma_S(a) = \sigma_S(a)a = 1\}$$
is contained in $\SIso(C,\sigma)(S)$.
\end{proposition}

\begin{proof}
For any $(c_i)_{i=0}^{n-1} \in C^n$ with $n \in \N$, we have
\begin{eqnarray*}
|a(c_0) \cdots a(c_{n-1})|_S \leq |a(c_0)|_S \cdots |a(c_{n-1})|_S \leq |c_0|_C \cdots |c_{n-1}|_C.
\end{eqnarray*}
\end{proof}

\begin{proposition}
\label{coincidence of SIso}
For any uniform commutative $R$-algebra $S$ with submultiplicative norm, we have
$$\SIso(C,\sigma)(S) = \{a \in C^{\vee} \otimes S \mid \forall c \in C, |a(c)|_S \leq |c|_C, a \sigma_S(a) = \sigma_S(a)a = 1\}.$$
\end{proposition}

\begin{proof}
The right hand side is contained in the left hand side by Proposition \ref{inclusion of SIso}. Let $a \in \SIso(C,\sigma)(S)$. By the definition of $\SIso(C,\sigma)(S)$, there exists some $R \in ]0,\infty[$ such that $|a(c_0) \cdots a(c_{n-1})|_S \leq R |c_0|_C \cdots |c_{n-1}|_C$ for any $n \in \N$ and $(c_i)_{i=0}^{n-1} \in C^n$. Let $c \in C$. For any $n \in \N$, we have $|a(c)|_S^n = |a(c)^n|_S \leq R |c|_C ^n$ by the uniformity of $S$, and hence $|a(c)|_S \leq R^{\frac{1}{n}} |c|_C$. It implies $|a(c)|_S \leq |c|_C$, as $\lim_{n \to \infty} R^{\frac{1}{n}} = 1$.
\end{proof}

We finish this subsection by giving sufficient conditions for the assumption in Proposition \ref{representability-short-isometries-comonoid}.

\begin{proposition}
\label{sufficient condition 1}
Let $C$ be a comonoid of $\Mod_R$. If the underlying module of $C$ over the underlying ring $R_0$ of $R$ is free of finite rank and every $R$-linear homomorphism $C \to R$ is Lipschitz, then for every $R$-algebra $S$, the morphism
$$\alpha_{C,S}:\uHom_R(C,R) \otimes_R S \to \uHom_S(C\otimes_R S,S)$$
is an isomorphism.
\end{proposition}

\begin{proof}
We denote by $R_0$ the underlying ring of $R$, and by $S_0$ the underlying ring of $S$. Take an $R_0$-linear basis $(c_i)_{i=0}^{n-1}$ with $n \in \N$ of the underlying $R_0$-module of $C$. Since every $R$-linear homomorphism $C \to R$ is Lipschitz, the underlying $R_0$-module of $C^{\vee}$ admits its dual basis $(\delta_i)_{i=0}^{n-1}$, and hence the underlying $S_0$-module of $C^{\vee} \otimes_R S$ is a free $S_0$-module of rank $n$ with the $S_0$-linear basis $(\delta_i \otimes 1)_{i=0}^{n-1}$. Put $C_0 \coloneqq \max_{i=0}^{n-1} |c_i|_C$ and $C_1 \coloneqq \max_{i=0}^{n-1} |\delta_i|_{C^{\vee}}$. The composite of the given morphism and the canonical embedding from $\uHom_S(C\otimes_R S,S)$ into the $S_0$-module of all $R$-linear homomorphisms $C \to S$ is an $S_0$-linear isomorphism, every $R$-linear homomorphism $C \to S$ extends to a Lipschitz $R$-linear homomorphism $C\otimes_R S \to S$. This implies that the given morphism is also an $S_0$-linear isomorphism. Therefore it suffices to show that the inverse of the given morphism is Lipschitz. For this purpose, it suffices to show that for any $d \in C^{\vee} \otimes_R S$, its image $d' \in \uHom_S(C\otimes_R S,S)$ is of norm $\geq |1|_S^{-1} C_0^{-1} C_1^{-1} C_S^{1-n} |d|_{C^{\vee} \otimes_R S}$.

\vspace{0.1in}
Take a unique $(s_i)_{i=0}^{n-1} \in S_0^n$ with $d = \sum_{i=0}^{n-1} \delta_i \otimes s_i$. We have
\begin{eqnarray*}
& & |d|_{C^{\vee} \otimes_R S} = \left| \sum_{i=0}^{n-1} \delta_i \otimes s_i \right|_{C^{\vee} \otimes_R S} \leq C_S^{n-1} \left\| (|\delta_i|_{C^{\vee}} |s_i|_S)_{i=0}^{n-1} \right\|_{\infty} \\
& \leq & C_S^{n-1} \left\| (|\delta_i|_{C^{\vee}})_{i=0}^{n-1} \right\|_{\infty} \left\| (|s_i|_S)_{i=0}^{n-1} \right\|_{\infty} = C_S^{n-1} C_1 \left\| (|s_i|_S)_{i=0}^{n-1} \right\|_{\infty}.
\end{eqnarray*}
On the other hand, we have
\begin{eqnarray*}
& & |d'|_{\uHom_S(C\otimes_R S,S)} = |d'|_{op} \geq \frac{|d'(c_i \times 1)|_S}{|c_i \otimes 1|_{C \otimes_R S}} \\
& = & \frac{|s_i|_S}{|c_i \otimes 1|_{C \otimes_R S}} \geq \frac{|s_i|_S}{|c_i|_C |1|_S} \geq |1|_S^{-1} C_0^{-1} |s_i|_S
\end{eqnarray*}
for any $0 \leq i \leq n-1$ and hence
\begin{eqnarray*}
& & |d'|_{\uHom_S(C\otimes_R S,S)} \geq |1|_S^{-1} C_0^{-1} \left\|(|s_i|_S)\right\|_{\infty} \geq |1|_S^{-1} C_0^{-1} C_1^{-1} C_S^{1-n} |d|_{C^{\vee} \otimes_R S}.
\end{eqnarray*}
\end{proof}

\begin{corollary}
\label{sufficient condition 2}
Let $C$ be a comonoid of $\Mod_R$. Suppose the underlying module of $C$ over the underlying ring $R_0$ of $R$ is free of rank $n \in \N$ with an $R_0$-linear basis $(c_i)_{i=1}^{n}$, and suppose further that there exists a constant $D \in ]0,\infty[$ such that for any $(r_i)_{i=1}^{n} \in R^n$, the equality $|(r_i)_{i=1}^{n}|_{S,\infty} \leq D |\sum_{i=1}^{n} r_i c_i|_C$ holds. Then the morphism
$$\alpha_{C,S}:\uHom_R(C,R) \otimes_R S \to \uHom_S(C\otimes_R S,S)$$
is an isomorphism for any $R$-algebra $S$.
\end{corollary}

\begin{proof}
The assertion immediately follows from Proposition \ref{criterion for boundedness} and Proposition \ref{sufficient condition 1}.
\end{proof}

%************************************************************************
\subsection{The global short isometry subgroup of $\GL_n$}
\label{short-isometries-GLn}
The aim of this whole Section is to prove the following theorem,
that is the main result of this paper.
\begin{theorem}
\label{main-theorem}
There exists an explicit Lipschitz coalgebra with involution $(C,\sigma)$
over $\Zbf=\Lip(\Z,|\cdot|_\infty,1)$ whose short isometry group, denoted
$K_n=\SIso(C,\sigma)$, is representable by a Lipschitz Banach halo $\Ac(K_n)$ and fulfils
$$K_n(R)\subset \GL_n(R)$$
for every $R\in \CAlg_\Zbf$ and such that, if we denote $\Rbf=\Lip(\R,|\cdot|_\infty,1)$ and $\Qbf_p=\Lip(\Q_p,|\cdot|_p,\infty)$,
we have
\begin{enumerate}
    \item $K_n(\Rbf)\cong O_n(\R)$,
    \item $K_n(\Qbf_p)\cong \GL_n(\Z_p)$,
    \item $K_n(\Zbf)\cong O_n(\R)\cap \GL_n(\Z)\equiv \GL_n(\F_{\{\pm 1\}})$
\end{enumerate}
\end{theorem}

\begin{remark}
We may change $\Zbf$ to $\sigma_t(\Zbf)$ for $t\in \R_+^*$ without changing
these results.
\end{remark}

We fix a positive integer $n$. We set $R = (\Z,|\cdot|_{\infty},1)$. Let $S$ be a commutative $R$-algebra with the underlying ring $S_0$ with $p_S = 1$. We fix a $q \in [1,\infty]$. We denote by $S^n_{\ell^q}$ the triad $(S_0^n,\|\cdot\|_{S,q},1)$ (cf.\ Definition \ref{abstract sum}), which forms an $S$-module by $1 \leq \min \{q,p_S\}$. We observe the structures of the endomorphism ring of $S^n_{\ell^q}$ (cf.\ proposition \ref{endomorphism ring}).

\begin{proposition}
\label{boundedness of dual}
\begin{itemize}
    \item [(i)] Every $S_0$-linear endomorphism of the $S$-modules $S^n_{\ell^q}$ and $\uHom_S(S^n_{\ell^q},S)$ is Lipschitz. In particular, for any $f \in \End_S(S^n_{\ell^q})$, its $S_0$-linear dual $f^{\vee}$ belongs to $\End_S(\uHom_S(S^n_{\ell^q},S))$.
    \item [(ii)] The map $\vee \colon \End_S(S^n_{\ell^q}) \to \End_S(\uHom_S(S^n_{\ell^q},S)), \ f \mapsto f^{\vee}$ is a Lipschitz $S_0$-linear isomorphism whose inverse is also Lipschitz.
\end{itemize}
\end{proposition}

\begin{proof}
The assertions immediately follow from Proposition \ref{criterion for boundedness}, because we can choose $E$ to be the canonical bases and $C$ to be $1$.
\end{proof}

We also denote by $\vee$ the inverse of $\vee$ in the assertion of Proposition \ref{boundedness of dual} (ii). We denote by $M(n,S,q)$ the direct product $\Lip(S)$-algebra
\begin{eqnarray*}
\End_{\Lip(S)}(\Lip(S^n_{\ell^q})) \times \End_{\Lip(S)}(\uHom_{\Lip(S)}(\Lip(S^n_{\ell^q}),\Lip(S))),
\end{eqnarray*}
and equip it with the involution
\begin{eqnarray*}
t_{M(n,S,q)} \colon M(n,S,q) & \to & M(n,S,q) \\
(f,g) & \mapsto & (g^{\vee},f^{\vee}),
\end{eqnarray*}
which is a monoid homomorphism $M(n,S,q) \to M(n,S,q)^{op}$. In particular, $M(n,S,q)$ is a monoid with involution over $\Lip(S)$, whose underlying $S_0$-algebra with involution represents $\GL_n \times_{\Spec(\Z)} \Spec(S_0)$.

We note that the underlying $S_0$-module of $M(n,S,q)$ is a free $S_0$-module of rank $2n^2$. Therefore its $S_0$-linear dual is also a free $S_0$-module of rank $2n^2$ and the multiplication of $M(n,S,q)$ induces a comultiplication on it. Contrary to Remark \ref{dual of monoid}, the comultiplication is Lipschitz as a map
\begin{eqnarray*}
\uHom_{\Lip(S)}(M(n,S,q),\Lip(S)) \to \uHom_{\Lip(S)}(M(n,S,q),\Lip(S))^{\otimes_S 2}
\end{eqnarray*}
by Proposition \ref{criterion for boundedness}. Therefore $\uHom_{\Lip(S)}(M(n,S,q),\Lip(S))$ forms a comonoid with involution over $\Lip(S)$ with respect to the dual involution of $t_{M(n,S,q)}$.

We denote by $(C(K_n),\sigma_{C(K_n)})$ the $\Zbf$-linear dual of the monoid with involution over $\Zbf$ given as $(M(n,R,2),t_{M(n,R,2)})$ except that the norm $|\cdot|_{M(n,R,2)}$ is replaced by the restriction of $|\cdot|_{M(n,(\R,|\cdot|_{\infty},1),2)}$ along
the natural map
$$
M(n,R,2)=M(n,(\Z,|\cdot|_\infty,1),2)\to
M(n,(\R,|\cdot|_\infty,1),2).
$$
Then $(C(K_n),\sigma_{C(K_n)})$ forms a comonoid with involution over $\Zbf$ again by Proposition \ref{criterion for boundedness}.

Since $(C(K_n),\sigma_{C(K_n)})$ is a comonoid object in $\Mod(\Zbf)$ with an involution, we can apply $\SIso$ to it.
The reader should be careful that the $\ell^2$ operator norm on matrices over $\Z$ might not be reflexive, i.e. might not coincide with the second dual norm. That is why we use the restriction of $|\cdot|_{M(n,(\R,|\cdot|_{\infty},1),2)}$ instead of $|\cdot|_{M(n,R,2)}$ here.

\begin{proposition}
The group of short isometries $\SIso(C(K_n),\sigma_{C(K_n)})$ is represented by $\Ac(K_n):=\Ac(\SIso(C(K_n),\sigma_{C(K_n)}))$.
\end{proposition}

\begin{proof}
The assertion immediately follows from Proposition \ref{representability-short-isometries-comonoid}. Indeed, we will show in Proposition \ref{criterion for the sufficient condition 2} that the dual basis of the canonical basis of $M(n,R,2)$ satisfies the assumption of Corollary \ref{sufficient condition 2} with respect to the constant $D = 1$.
\end{proof}

We will give an explicit computation of $|\cdot|_{C(K_n)}$ for the reader's convenience in \S \ref{Computation of $p$-adic and real points}.

%************************************************************************
\subsection{Global short isometry subgroups of other classical groups}
\label{Global short isometry subgroups of other classical groups}

Let $V=\Z^n$ be the free module, that we first consider as a mere module.
Recall from \cite{Involutions} the definition of the group of isometries of an algebra with involution $(A,\sigma)$ over a base commutative ring $R$, i.e.\ a pair of an $R$-algebra $A$ and an involution $\sigma$ on $A$ that is a ring homomorphism $A \to A^{op}$.

\begin{definition}
Let $(A,\sigma)$ be an algebra with involution over a commutative ring $R$ and
$S$ be a commutative $R$-algebra. Then the group of isometries of $(A,\sigma)$
with values in $S$ is
$$
\Iso(A,\sigma)(S):=\{a\in A\otimes_R S \mid a\sigma(a)=\sigma(a)a=1_A\}.
$$
\end{definition}
By abuse of notation, if $R$, $S$ and $A$ are Lipschitz Banach halos, we will also
denote $\Iso(A,\sigma)(S)$ the corresponding multiplicative subgroup of the scalar extension
$A\otimes_R S$. By Proposition \ref{inclusion of SIso} and Proposition \ref{coincidence of SIso}, we obtain the following justification of the terminology of ``the functor of short isometries'' in \S \ref{Short isometries of Lipschitz coalgebras}:

\begin{proposition}
\label{coincidence of Iso}
Let $R$ be a commutative Lipschitz Banach halo with $D_R = 1$, $(C,\sigma)$ a coalgebra with involution over $R$, and $S$ a commutative $R$-algebra. We have
$$
\{a\in \Iso(C^{\vee},\sigma^{\vee})(S) \mid \forall c \in C, |a(c)|_S \leq |c|_C\} \subset \SIso(C,\sigma)(S).
$$
In addition, if $S$ is uniform, then the inclusion is an equality.
\end{proposition}

We already know that $\GL_n$ is
naturally identified with the isometry group $\Iso(A,\sigma)$ of an algebra with involution given
by $A=\End(V)\times \End(V^\vee)$ with its transposition involution
$$
\sigma(f,g)=(g^\vee,f^\vee),
$$
where $\End$ denotes the algebraic endomorphism ring of a mere module.
Let $\phi:V\times V\to \Z$ be a non-degenerate bilinear pairing and consider
the associated involution algebra $(\End(V),\sigma_\phi)$,
where
$$\sigma_\phi(f)=(\phi^\flat)^{-1}\circ f^\vee\circ \phi^\flat,$$
with $\phi^\flat:V\overset{\sim}{\longrightarrow} V^\vee$ the isomorphism induced by $\phi$.

There is a natural injective morphism of involution algebras
$$
\begin{array}{ccc}
\End(V)     & \to       & \End(V)\times \End(V^\vee)\\
f           & \mapsto   & (f,[\sigma_\psi(f)]^\vee)
\end{array}
$$
that is compatible with the involutions $\sigma$ and $\sigma_\phi$ and
induces the natural map
$$\Iso(\End(V),\sigma_\phi)\to \GL_n$$
between the classical group associated to the given non-degenerate
bilinear form and $\GL_n$.

By $\Z$-linear duality, we get a surjective morphism of comonoid with involution
$$
[\End(V)\times \End(V^\vee)]^\vee\to \End(V)^\vee.
$$
Since $\End(V)\times \End(V^\vee)$ is the underlying monoid with involution over $\Z$ of $M(2,(\Z,|\cdot|_{\infty},1,2)$, the source of the morphism precisely coincides with the underlying comonoid with involution over $\Z$ of the comonoid with involution $C(K_n)$ over $\Mod(\Zbf)$. We use this map to induce
a quotient norm $\|\cdot\|_{\phi}$ on the target comonoid with involution $(\End(V)^\vee,\sigma_\phi^\vee)$ over $\Z$.

This gives a comonoid $C(K_n(\phi)) = (\End(V)^{\vee},\|\cdot\|_{\phi},(2,1))$ with involution $\sigma_\phi^\vee$ over $\Zbf$.

\begin{definition}
The short isometry group $\SIso(C(K_n(\phi)),\sigma_\phi^\vee)$ of the comonoid with involution
$(C(K_n(\phi)),\sigma_\phi^\vee)$ is denoted $K_n(\phi)$ and
called the short isometry group of the given bilinear form.
\end{definition}

The short surjective comonoid homomorphism
$$C(K_n) \to C(K_n(\phi))$$
given by the quotient map
$$[\End(V) \times \End(V^{\vee})]^{\vee} \to \End(V)^{\vee}$$
induces a natural embedding
$$
K_n(\phi) \to K_n \subset \GL_n
$$
because of the compatibility of $\sigma_{C(K_n)}$ and $\sigma_\phi$. By Proposition \ref{coincidence of SIso}, this embedding induces a natural embedding
\begin{eqnarray*}
& & K_n(\phi)(S) = \SIso(C(K_n(\phi)),\sigma_\phi^\vee)(S) = \SIso(\End(V)^{\vee},\|\cdot\|_{\phi},\sigma_\phi^\vee)(S) \\
& \to & \Iso(\End(V)^{\vee \vee},\sigma_\phi^{\vee \vee})(S) \cong \Iso(\End(V),\sigma_\phi)(S)
\end{eqnarray*}
of subgroups of $\GL_n(S)$ for any uniform $\Zbf$-algebra $S$.

\begin{remark}
More generally, let $B$ be a ring which is a free $\Z$-module of finite rank $m$, and $\gamma$ be an involution of $B$ that is a ring homomorphism $B \to B^{op}$. Fix a left $B$-module $M$ which is a free $\Z$-module of finite rank $n$ such that the canonical morphism $B \to \End(M)$ is injective. We have considered the specific setting $(B,M) = (\End(V),V)$, but $B$ can be chosen arbitrarily because $M$ can be taken as the regular left $B$-module.

Then we may
embed $B$ in
$$A=\End(M)\times \End(M^\vee)$$
by the left action on $M$ and the dual action on
$M^\vee$ twisted by $\gamma$, i.e.\, by $(b\cdot f)(m):=f(\gamma(b) m)$ for $(b,f,m)\in B \times M^\vee \times M$.
A construction similar to the above one gives a quotient comonoid map
$$
A^\vee\to B^\vee
$$
that allows us to define a subgroup
$$
K(B,\gamma) \subset K_n
$$
of the standard subgroup $K_n$ of $\GL_n\cong \GL(M)$ contained in $\Iso(B,\gamma)$.
\end{remark}

%************************************************************************
\subsection{Computation of $p$-adic and real points}
\label{Computation of $p$-adic and real points}

Following the convention in \S \ref{Global short isometry subgroups of other classical groups}, we demonstrate the computation of $p$-adic and real points of $K_n$. Throughout this subsection, let $R$ denote the short Banach halo $(\Z,|\cdot|_{\infty},1)$, and $S$ denote a commutative $R$-algebra with undering ring $S_0$ and $p_S = 1$.

For an $A \in \End(V)$ (resp.\ $\End(V^{\vee})$) and an $(i,j) \in \N^2$ with $1 \leq i,j \leq n$, we denote by $A[i,j]$ the $(i,j)$-entry of $A$ with respect to the matrix presentation for the canonical basis of $V$ (resp.\ $V^{\vee}$). For an $(i,j) \in \N^2$ with $1 \leq i,j \leq n$, we denote by $E_{i,j} \in \End_{\Z}(\Z^n)$ the matrix whose $(i,j)$-entry is $1$ and whose other entries are $0$. For any $(i,j,b) \in \N^3$ with $1 \leq i,j \leq n$ and $b \leq 1$, we denote by $E_{i,j,b} \in M(n,R,2)$ the image of $E_{i,j}$ in $\End_R(R^n_{\ell^2})$ when $b = 0$ and in $\End_R((R^n_{\ell^2})^{\vee})$ when $b = 1$. Then $((E_{i,j,b})_{i,j=1}^{n})_{b=0}^{1}$ forms a $\Z$-linear basis of the underlying $\Z$-module of $M(n,R,2)$. We denote by $((e_{i,j,b})_{i,j=1}^{n})_{b=0}^{1}$ the $\Z$-linear basis of the underlying $\Z$-module of $C(K_n)$ dual to $((E_{i,j,b})_{i,j=1}^{n})_{b=0}^{1}$.

\begin{proposition}
\label{norm of e_i,j,b}
For any $(i,j,b) \in \N^3$ with $1 \leq i,j \leq n$ and $b \leq 1$, the equality $|e_{i,j,b}|_{C(K_n)} = 1$ holds.
\end{proposition}

\begin{proof}
Set $S = (\R,|\cdot|_{\infty},1)$. For any $(B_0,B_1) \in M(n,R,2)$ naturally regarded as an element of $M(n,S,2)$, we have
$$|e_{i,j,b}(B_0,B_1)|_{\infty} = |B_b[i,j]|_{\infty} \leq |(B_0,B_1)|_{M(n,S,2)}.$$
This implies $|e_{i,j,b}|_{\Cc(K_n)} \leq 1$. On the other hand, we have $|e_{i,j,b}(E_{i,j,b})|_{\infty} = |1|_{\infty} = 1$. This implies $|e_{i,j,b}|_{C(K_n)} \geq 1$ by $|E_{i,j,b}|_{M(n,S,2)} = 1$.
\end{proof}

\begin{proposition}
\label{criterion for the sufficient condition 2}
For any $((r_{i,j,b})_{i,j=1}^{n})_{b=0}^{1} \in ((R^n)^n)^2$, the equality $\|((|r_{i,j,b}|)_{i,j=1}^{n})_{b=0}^{1}\|_{\infty}\| \leq |\sum_{b=0}^{1} \sum_{i,j=1}^{n} r_{i,j,b} e_{i,j,b}|_{C(K_n)}$ holds.
\end{proposition}

\begin{proof}
Set $S = (\R,|\cdot|_{\infty},1)$. For any $(i_0,j_0,b_0) \in \N^3$ with $1 \leq i_0,j_0 \leq n$ and $b_0 \leq 1$, we have
\begin{eqnarray*}
\left| \left( \sum_{b=0}^{1} \sum_{i,j=1}^{n} r_{i,j,b} e_{i,j,b} \right)(E_{i_0,j_0,b_0}) \right|_R = |r_{i_0,j_0,b_0}|_R.
\end{eqnarray*}
and hence
\begin{eqnarray*}
\left| \sum_{b=0}^{1} \sum_{i,j=1}^{n} r_{i,j,b} e_{i,j,b} \right|_R \geq \frac{|r_{i_0,j_0,b_0}|_R}{|E_{i_0,j_0,b_0}|_{M(n,S,2)}} = |r_{i_0,j_0,b_0}|_R.
\end{eqnarray*}
This implies the assertion.
\end{proof}

Let $(f,s) \in C(K_n)^{\vee} \times S$. We denote by $i_S(f,s)_0 \in \End_{S_0}(S_0^n)$ the $S_0$-linear endomorphism
\begin{eqnarray*}
S_0^n & \to & S_0^n \\
(t_j)_{j=1}^{n} & \mapsto & \left( \sum_{j=1}^{n} s f(e_{i,j,b}) t_j \right)_{i=1}^{n},
\end{eqnarray*}
and by $i_S(f,s)_1 \in \End_{S_0}(\uHom_{S_0}(S_0^n,S_0))$ the $S_0$-linear endomorphism
\begin{eqnarray*}
\uHom_{S_0}(S_0^n,S_0) & \to & \uHom_{S_0}(S_0^n,S_0) \\
u & \mapsto & \left( (t_i)_{i=1}^{n} \mapsto \sum_{i=1}^{n} s t_i u \left( (f(e_{i,j,b}))_{j=1}^{n} \right) \right).
\end{eqnarray*}
We put $i_S(f,s) = (i_S(f,s)_0,i_S(f,s)) \in \End_{S_0}(S_0^n) \times \End_{S_0}(\uHom_{S_0}(S_0^n,S_0))$. By Proposition \ref{abstract sum}, we have $i_S(f,s) \in M(n,S,q)$ for any $q \in [1,\infty]$.

\begin{proposition}
\label{boundedness of i_S}
For any $(f,s,q) \in C(K_n)^{\vee} \times S_0 \times [1,\infty]$, the inequality $|i_S(f,s)|_{M(n,S,q)} \leq n^{1+\frac{1}{q}} |f|_{C(K_n)^{\vee}} |s|_S$ holds.
\end{proposition}

\begin{proof}
For any $(t_i)_{i=1}^{n} \in S_0^n$, we have
\begin{eqnarray*}
& & \|i_S(f,s)_0((t_i)_{i=1}^{n})\|_{S,q} = \left\|\left(\sum_{j=1}^{n} s f(e_{i,j,0}) t_j \right)_{i=1}^{n} \right\|_{S,q} \\
&=& \left\|\left(\left|\sum_{j=1}^{n} s f(e_{i,j,0}) t_j \right|_S \right)_{i=1}^{n} \right\|_{q} \leq \left\|(\|(|s f(e_{i,j,0}) t_j|_S)_{j=1}^{n}\|_{1})_{i=1}^{n}\right\|_{q} \\
&\leq& \left\|(\|(|s|_S |f|_{op} |e_{i,j,0}|_{\Cc(K_n)} |t_j|_S)_{j=1}^{n}\right\|_{1})_{i=1}^{n}\|_{q} \\
&\leq& |s|_S |f|_{op} \left\|(\|(|t_j|_S)_{j=1}^{n}\|_{1})_{i=1}^{n}\right\|_{q} = n^{\frac{1}{q}} |s|_S |f|_{op} \|(|t_j|_S)_{j=1}^{n}\|_{1} \\
&\leq& n^{1+\frac{1}{q}} |s|_S |f|_{op} \|(|t_j|_S)_{j=1}^{n}\|_{{\infty}} \leq n^{1+\frac{1}{q}} |s|_S |f|_{op} \|(|t_j|_S)_{j=1}^{n}\|_{q} \\
&=& n^{1+\frac{1}{q}} |s|_S |f|_{\Cc(K_n)^{\vee}} \|(t_j)_{j=1}^{n}\|_{S,q}
\end{eqnarray*}
by $p_S = 1$ and Proposition \ref{norm of e_i,j,b}. This implies
\begin{eqnarray*}
|i_S(f,s)_0|_{\End_S(S^n_{\ell^q})} \leq n^{1+\frac{1}{q}} |f|_{\Cc(K_n)^{\vee}} |s|_S.
\end{eqnarray*}
For any $u \in \uHom_{S_0}(S_0^n,S_0)$ and $(t_i)_{i=1}^{n} \in S_0^n$, we have
\begin{eqnarray*}
& & |i_S(f,s)_1(u)((t_i)_{i=1}^{n})|_S = \left|\left(\sum_{i=1}^{n} s t_i u \left( (f(e_{i,j,1}))_{j=1}^{n} \right) \right) \right|_S \\
&\leq& \left\|\left( \left| s t_i u \left( (f(e_{i,j,1}))_{j=1}^{n} \right) \right|_S \right)_{i=1}^{n} \right\|_{1} \\
&\leq& |s|_S |u|_{op} \left\|\left( |t_i|_S \left\|(f(e_{i,j,1}))_{j=1}^{n}\right\|_{S,q} \right)_{i=1}^{n} \right\|_{1} \\
&\leq& |s|_S |u|_{op} \left\|\left( |t_i|_S \left\|(|f(e_{i,j,1})|_S)_{j=1}^{n}\right\|_{q} \right)_{i=1}^{n} \right\|_{1} \\
&\leq& |s|_S |u|_{op} \left\|\left( |t_i|_S \left\|(|f|_{op} |e_{i,j,1}|_{\Cc(K_n)})_{j=1}^{n}\right\|_{q} \right)_{i=1}^{n} \right\|_{1} \\
&\leq& |s|_S |u|_{op} \left\|\left( |t_i|_S \left\|(|f|_{op})_{j=1}^{n}\right\|_{q} \right)_{i=1}^{n} \right\|_{1} \leq n^{\frac{1}{q}} |s|_S |u|_{op} |f|_{op} \left\| (|t_i|_S)_{i=1}^{n} \right\|_{1} \\
&\leq& n^{1+\frac{1}{q}} |s|_S |u|_{op} |f|_{op} \left\| (|t_i|_S)_{i=1}^{n} \right\|_{{\infty}} \leq n^{1+\frac{1}{q}} |s|_S |u|_{op} |f|_{op} \left\| (|t_i|_S)_{i=1}^{n} \right\|_{q} \\
&=& n^{1+\frac{1}{q}} |s|_S |f|_{\Cc(K_n)^{\vee}} |u|_{\uHom_S(S^n_{\ell^q},S)} \left\| (t_i)_{i=1}^{n} \right\|_{S,q}
\end{eqnarray*}
again by $p_S = 1$ and Proposition \ref{norm of e_i,j,b}, and hence $|i_S(f,s)_1(u)|_{\uHom_S(S^n_{\ell^q},S)} \leq n^{1+\frac{1}{q}} |s|_S |f|_{\Cc(K_n)^{\vee}} |u|_{\uHom_S(S^n_{\ell^q},S)}$. This implies
\begin{eqnarray*}
|i_S(f,s)_1|_{\End_S(S^n_{\ell^q})} \leq n^{1+\frac{1}{q}} |f|_{\Cc(K_n)^{\vee}} |s|_S.
\end{eqnarray*}
Thus we obtain $|i_s(f,s)|_{M(n,S,q)} \leq n^{1+\frac{1}{q}} |f|_{C(K_n)^{\vee}} |s|_S$.
\end{proof}

For any $q \in [1,\infty]$, we obtain the Lipschitz $R$-bilinear map
\begin{eqnarray*}
i_S \colon C(K_n)^{\vee} \otimes^m S & \to & M(n,S,q) \\
(f,t) & \mapsto & i_S(f,t)
\end{eqnarray*}
whose Lipschitz constant is bounded by $n^{1+\frac{1}{q}}$ by Proposition \ref{boundedness of i_S}, and it induces a bounded $S_0$-linear homomorphism
\begin{eqnarray*}
\iota_{S,q} \colon \Cc(K_n)^{\vee} \otimes \Lip(S) \to M(n,S,q)
\end{eqnarray*}
whose Lipschitz constant is bounded by $n^{1+\frac{1}{q}}$ by Proposition \ref{universal-property-tensor-product} applied to $i_S$.

\begin{proposition}
\label{real orthogonal group}
When $S = (\R,|\cdot|_{\infty},1)$, the restriction of $\iota_{S,2}$ to
$$\SIso(C(K_n),\sigma_{C(K_n)})(\Rbf) \subset C(K_n)^{\vee} \otimes \Rbf = C(K_n)^{\vee} \otimes \Lip(S)$$
is a bijective map onto $\{(U,U^{-1}) \mid U \in O_n(\R)\}$.
\end{proposition}

In order to show Proposition \ref{real orthogonal group}, we prepare lemmata.

\begin{lemma}
\label{dual of lattice}
Set $S = (\R,|\cdot|_{\infty},1)$. Let $M$ be an $S$-module $M$ whose underlying $\R$-vector space is of dimension $d \in \N$, and $V \subset M$ an $R$-submodule whose underlying $\Z$-module is a free $\Z$-module of rank $d$. Then the natural embedding $V^{\vee} \to \uHom_S(M,S)$ is isometric.
\end{lemma}

\begin{proof}
We denote by $i$ the natural embedding in the assertion. We show $||f|_{op} - |i(f)|_{op}| < \epsilon$ for any $(f,\epsilon) \in V^{\vee} \times ]0,\infty[$. By the definition of $|f|_{op}$, there exists a $v \in V$ such that $|f(v)|_V > (|f|_{op} - \epsilon)|v|_V$. We obtain
\begin{eqnarray*}
|i(f)(v)|_M = |f(v)|_V > (|f|_{op} - \epsilon)|v|_V = (|f|_{op} - \epsilon)|v|_M.
\end{eqnarray*}
This implies $|i(f)|_{op} > |f|_{op} - \epsilon$. Let $m \in M$. Since the image of $V \otimes_{\Z} \Q$ is dense in $M$, there exists a $(v,s) \in V \times (\Z \setminus \{0\})$ such that $|m - s^{-1}v| < 2^{-1} \min \{|f|_{op}^{-1},|i(f)|_{op}^{-1}\} |m|_M^{-1} \epsilon$. We have
\begin{eqnarray*}
i(f)(m) = i(f)(s^{-1}v) + i(f)(m - s^{-1}v) = s^{-1} f(v) + i(f)(m - s^{-1}v)
\end{eqnarray*}
and hence
\begin{eqnarray*}
& & |i(f)(m)|_S = |s^{-1} f(v) + i(f)(m - s^{-1}v)|_S \leq |s^{-1}|_S |f|_{op} |v|_V + |i(f)|_{op} |m - s^{-1}v|_S \\
&<& |s^{-1}|_S |f|_{op} |v|_V + 2^{-1} \epsilon = |s^{-1}|_S |f|_{op} |v|_M + 2^{-1} |m|_M \epsilon = |f|_{op} |s^{-1}v|_M + 2^{-1} |m|_M \epsilon \\
&\leq& |f|_{op}(|m|_M + |m-s^{-1}v|_M) + 2^{-1} |m|_M \epsilon\\
&<& |f|_{op} |m|_M + 2^{-1} |m|_M \epsilon + 2^{-1} |m|_M \epsilon = (|f|_{op} + \epsilon) |m|_M.
\end{eqnarray*}
It implies $|i(f)|_{op} \leq |f|_{op} + \epsilon$.
\end{proof}

\begin{lemma}
\label{embedding into the dual of M(n,S,2)}
When $S = (\R,|\cdot|_{\infty},1)$, the natural embedding $C(K_n) \to \uHom_S(M(n,S,2),S)$ is an isometry.
\end{lemma}

\begin{proof}
The assertion immediately follows from Lemma \ref{dual of lattice} by the definition of $|\cdot|_{\Cc(K_n)}$ using the restriction of $|\cdot|_{M(n,S,2)}$ to $M(n,R,2)$.
\end{proof}

\begin{lemma}
\label{embedding into the second dual of M(n,S,2)}
When $S = (\R,|\cdot|_{\infty},1)$, the natural embedding $C(K_n)^{\vee} \to \uHom_S(\uHom_S(M(n,S,2),S),S)$ is an isometry.
\end{lemma}

\begin{proof}
The assertion immediately follows from Lemma \ref{dual of lattice} and Lemma \ref{embedding into the dual of M(n,S,2)}.
\end{proof}

\begin{lemma}
\label{isometry of restriction of i_real}
When $S = (\R,|\cdot|_{\infty},1)$, the composite of the scalar extension $C(K_n)^{\vee} \to C(K_n)^{\vee} \otimes \Rbf = C(K_n)^{\vee} \otimes \Lip(S)$ and
$$\iota_{S,2} \colon \Cc(K_n)^{\vee} \otimes \Lip(S) \to M(n,S,2)$$ is an isometry.
\end{lemma}

\begin{proof}
By Lemma \ref{embedding into the second dual of M(n,S,2)}, the natural embedding
\begin{eqnarray*}
i_0 \colon \Cc(K_n)^{\vee} \to \uHom_S(\uHom_S(M(n,S,2),S),S)
\end{eqnarray*}
is an isometry. We denote by $i_1$ the composite of the scalar multiplication $C(K_n)^{\vee} \to C(K_n)^{\vee} \otimes \Rbf$ and $\iota_{S,2}$. Since $i_0$ is an isometry and coincides with the composite of $i_1$ and the second dual embedding $M(n,S,2) \to \uHom_S(\uHom_S(M(n,S,2),S),S)$, which is an isometry by the classical real analysis, $i_1$ is an isometry.
\end{proof}

\begin{lemma}
\label{presentation of tensor}
When $S = (\R,|\cdot|_{\infty},1)$ or $S = (\Q_p,|\cdot|_p,\infty)$, then for any $a \in C(K_n)^{\vee} \otimes \Lip(S)$, the equality
$$
\begin{array}{c}
|a|_{C(K_n)^{\vee} \otimes \Lip(S)}\\
\Vert\\
\inf \left\{\|(|f_k|_{C(K_n)^{\vee}} |s_k|_{S})_{k=1}^{K}\|_{\{1,\ldots,K\},p_{\Rbf}}
\:
\begin{array}{|l}
K \in \N, (f_k)_{k=1}^{K} \in (C(K_n)^{\vee})^K,\\
(s_k)_{k=1}^{K} \in S^K, a = \sum_{k=1}^{K} f_k \otimes s_k
\end{array}\right\}
\end{array}
$$
holds.
\end{lemma}

\begin{proof}
In both cases, the composite of the inclusion
$$\Z^{\oplus (C(K_n)^{\vee} \times S)} \hookrightarrow \ell^{p_S}(C(K_n)^{\vee} \otimes^m S,R_{p_S})$$
and the canonical projection
$$\ell^{p_S}(C(K_n)^{\vee} \otimes^m S,R)_{p_S} \twoheadrightarrow C(K_n)^{\vee} \otimes \Lip(S)$$
has the dense image by construction, and hence is surjective by the classical fact that every subspace of a finite dimensional normed $S$-vector space is closed. Therefore the assertion follows from the definition of the quotient norm.
\end{proof}

\begin{lemma}
\label{isometry of i_real}
When $S = (\R,|\cdot|_{\infty},1)$, the morphisme
$$\iota_{S,2} \colon \Cc(K_n)^{\vee} \otimes \Lip(S) \to M(n,S,2)$$
is an isometric isomorphism.
\end{lemma}

\begin{proof}
Let $a \in C(K_n)^{\vee} \otimes \Rbf$. Let $\sum_{k=1}^{K} f_k \otimes s_k$ be a presentation of $a$ with $K \in \N$. By Lemma \ref{isometry of restriction of i_real}, we have
\begin{eqnarray*}
& & \|(|f_k|_{C(K_n)^{\vee}} |s_k|_S)_{k=1}^{K}\|_{\{1,\ldots,K\},p_{\Rbf}} = \|(|f_k|_{op} |s_k|_S)_{k=1}^{K}\|_{\{1,\ldots,K\},(2,1)} \\
&\leq& \|(|f_k|_{op} |s_k|_S)_{k=1}^{K}\|_{1} = \|(|\iota_{S,2}(f_k \otimes 1)|_{M(n,S,2)} |s_k|_S)_{k=1}^{K}\|_{1}\\
&=& \|(|\iota_{S,2}(f_k \otimes s_k)|_{M(n,S,2)})_{k=1}^{K}\|_{1} \geq \left|\sum_{k=1}^{K} \iota_{S,2}(f_k,s_k) \right|_{M(n,S,2)}\\
&=& \left|\iota_{S,2} \left(\sum_{k=1}^{K} f_k \otimes s_k \right) \right|_{M(n,S,2)} = |\iota_{S,2}(a)|_{M(n,S,2)}.
\end{eqnarray*}
By Lemma \ref{presentation of tensor}, we obtain $|a|_{C(K_n)^{\vee} \otimes \Rbf} \geq |\iota_{S,2}(a)|_{M(n,S,2)}$.

\vspace{0.1in}
Assume $|a|_{C(K_n)^{\vee} \otimes \Rbf} > |\iota_{S,2}(a)|_{M(n,S,2)}$ and set
$$\epsilon := 2^{-1}(|a|_{C(K_n)^{\vee} \otimes \Rbf} - |\iota_{S,2}(a)|_{M(n,S,2)}) > 0.$$
Since the algebraic tensor product $C(K_n)^{\vee} \otimes_{\Z} \Q$ is dense in $C(K_n)^{\vee} \otimes \Rbf$, there exists an $(f,s) \in C(K_n)^{\vee} \times \Z$ with $s \neq 0$ and $|a - f \otimes s^{-1}|_{C(K_n)^{\vee} \otimes \Rbf} < \epsilon$. By Lemma \ref{isometry of restriction of i_real}, we have $|\iota_{S,2}(f \otimes 1)|_{M(n,S,2)} = |f|_{C(K_n)^{\vee}}$. We obtain
\begin{eqnarray*}
& & |a|_{C(K_n)^{\vee} \otimes \Rbf} \leq |f \otimes s^{-1}|_{C(K_n)^{\vee} \otimes \Rbf} + \epsilon\\
&=& |s^{-1}|_{\infty} |f \otimes 1|_{C(K_n)^{\vee} \otimes \Rbf} + \epsilon \leq |s^{-1}|_{\infty} |f|_{op} |1|_S + \epsilon\\
&=& |s^{-1}|_{\infty} |f|_{C(K_n)^{\vee}} + \epsilon = |s^{-1}|_{\infty} |\iota_{S,2}(f \otimes 1)|_{M(n,S,2)} + \epsilon \\
&=& |\iota_{S,2}(f \otimes s^{-1})|_{M(n,S,2)} + \epsilon \\
&=& |\iota_{S,2}(a)|_{M(n,S,2)} + |\iota_{S,2}(a - f \otimes s^{-1})|_{M(n,S,2)} + \epsilon \\
&\leq& |\iota_{S,2}(a)|_{M(n,S,2)} + |a - f \otimes s^{-1}|_{C(K_n)^{\vee} \otimes \Rbf} + \epsilon \leq |\iota_{S,2}(a)|_{M(n,S,2)} + 2 \epsilon,
\end{eqnarray*}
but it contradicts the definition of $\epsilon$. It implies $|a|_{C(K_n)^{\vee} \otimes \Rbf} = |\iota_{S,2}(a)|_{M(n,S,2)}$. Therefore $\iota_{S,2}$ is isometric.
By
$$
\begin{array}{ccl}
\dim_{\R}(C(K_n)^{\vee} \otimes \Rbf) &= & \dim_{\R}(\End(\Z^n)^{\vee \vee} \otimes_{\Z} \R) = 2n^2\\
& =&  \dim_{\R}(\End_{\R}(\R^n)^2) = \dim_{\R}(M(n,S,2))
\end{array}
$$
we get that $\iota_{S,2}$ is bijective.
\end{proof}

\begin{lemma}
\label{closed unit ball of real counterpart}
When $S = (\R,|\cdot|_{\infty},1)$, an $a \in C(K_n)^{\vee} \otimes \Rbf$ satisfies $|a(c)|_S \leq |c|_{C(K_n)}$ for any $c \in C(K_n)$ if and only if $|\iota_{S,2}(a)|_{M(n,S,2)} \leq 1$.
\end{lemma}

\begin{proof}
First, suppose $|\iota_{S,2}(a)|_{M(n,S,2)} \leq 1$. We show $|a(c)|_S \leq |c|_{C(K_n)}$ for any $c \in C(K_n)$. By Lemma \ref{isometry of i_real}, we have $|a|_{{C(K_n)}^{\vee} \otimes \Rbf} \leq 1$. Let $\epsilon > 0$. By Lemma \ref{presentation of tensor}, there exists a presentation $\sum_{k=1}^{K} f_k \otimes s_k$ with $K \in \N$ of $a$ satisfying $\|(|f_k|_{C(K_n)^{\vee}} |s_k|_S)_{k=1}^{K}\|_{\{1,\ldots,K\},p_{\Rbf}} < 1+\epsilon$. We have
$$
\begin{array}{ccl}
& & |a(c)|_S = \left|\sum_{k=1}^{K} s_k f_k(c) \right|_S \leq \|(|s_kf_k(c)|_{\infty})_{k=1}^{K}\|_{\{1,\ldots,K\},p_{\Rbf}}\\
&\leq& \|(|f_k|_{C(K_n)^{\vee}} |s_k|_{\infty} |c|_{C(K_n)})_{k=1}^{K}\|_{\{1,\ldots,K\},p_{\Rbf}} \\
&=& \|(|f_k|_{C(K_n)^{\vee}} |s_k|_{\infty})_{k=1}^{K}\|_{\{1,\ldots,K\},p_{\Rbf}} |c|_{C(K_n)}\\
&\leq& (1+\epsilon) |c|_{C(K_n)}.
\end{array}
$$
It implies $|a(c)|_S \leq |c|_{C(K_n)}$.

\vspace{0.1in}
Next, suppose $|\iota_{S,2}(a)|_{M(n,S,2)} > 1$. Since every normed $\R$-vector space isometrically embeds into its second dual, there exists an $A \in \uHom_S(M(n,S,2),S)$ such that $|A(\iota_{S,2}(a))|_{\infty} > |A|_{\uHom_S(M(n,S,2),S)}$. Put
$$\epsilon := 2^{-1}(A(\iota_{S,2}(a)) - |A|_{\uHom_S(M(n,S,2),S)}) > 0.$$
Take an $A' \in \uHom_{\Q}(\End_{\Q}(\Q^n) \times \End_{\Q}(\uHom_{\Q}(\Q^n,\Q)),\Q)$ naturally regarded as an element of $\uHom_S(M(n,S,2),S)$ by the canonical bases of $S^n_{\ell^2}$ and $\uHom_S(S^n_{\ell^2},S)$ satisfying the following:
\begin{itemize}
\item[(i)] $|(A-A')(\iota_{S,2}(a))|_{\infty} < \epsilon$
\item[(ii)] $|A-A'|_{\uHom_S(M(n,S,2),S)} < \epsilon$
\end{itemize}
Take an $s \in \Z \setminus \{0\}$ such that $c := sA'$ is contained in the image of $[\End(V) \times \End(V^{\vee})]^{\vee}$, which is the underlying $\Z$-module of $C(K_n)$. By Lemma \ref{embedding into the dual of M(n,S,2)}, we have $|c|_{C(K_n)} = |s|_{\infty} |A'|_{\uHom_S(M(n,S,2),S}$. Put $c = \sum_{b=0}^{1} \sum_{i,j=1}^{n} B_{i,j,b} e_{i,j,b}$ with $((B_{i,j,b})_{i,j=1}^{n})_{b=0}^{1} \in (((\Z^n)^n)^2$. We obtain
\begin{eqnarray*}
& & a(c) = a \left(\sum_{b=0}^{1} \sum_{i,j=1}^{n} B_{i,j,b} e_{i,j,b} \right) = \sum_{b=0}^{1} \sum_{i,j=1}^{n} B_{i,j,b} a \left(e_{i,j,b} \right) \\
& = & s \sum_{b=0}^{1} \sum_{i,j=1}^{n} s^{-1} B_{i,j,b} a(e_{i,j,b}) \\
& = & s \sum_{b=0}^{1} \sum_{i,j=1}^{n} s^{-1} B_{i,j,b} (e_{i,j,b} \otimes 1_{\Rbf})(i_{S,2}(a)) \\
& = & s \left( \sum_{b=0}^{1} \sum_{i,j=1}^{n} B_{i,j,b} e_{i,j,b} \otimes 1_{\Rbf} s^{-1} \right) (i_{S,2}(a)) \\
& = & s A'(\iota_{S,2}(a)) \\
\end{eqnarray*}
and hence
\begin{eqnarray*}
& & |a(c)|_{\infty} = |s A'(\iota_{S,2}(a))|_{\infty} = |s|_{\infty} |A'(\iota_{S,2}(a))|_{\infty} \\
&>& |s|_{\infty}(|A(\iota_{S,2}(a))|_{\infty} - \epsilon) = |s|_{\infty}(|A|_{\uHom_S(M(n,S,2),S)} + \epsilon) \\
&>& |s|_{\infty} |A'|_{\uHom_S(M(n,S,2),S)} = |c|_{\uHom_S(M(n,S,2),S)} = |c|_{C(K_n)}.
\end{eqnarray*}
Thus there exists a $c \in C(K_n)$ satisfying $|a(c)|_S > |c|_{C(K_n)}$.
\end{proof}

\begin{proof}[Proof of Proposition \ref{real orthogonal group}]
By
\begin{eqnarray*}
O_n(\R) = \{U \in \GL_n(\R) \mid |U|_{\End_S(S^2_{\ell^2})} = |U^{-1}|_{\End_S(S^2_{\ell^2})} = 1\},
\end{eqnarray*}
the assertion immediately follows from Proposition \ref{coincidence of SIso}, Lemma \ref{isometry of i_real}, and Lemma \ref{closed unit ball of real counterpart}.
\end{proof}

\begin{proposition}
\label{p-adic orthogonal group}
When $S = (\Q_p,|\cdot|_{p},\infty)$, the restriction of $$\iota_{S,\infty} \colon \Cc(K_n)^{\vee} \otimes \Lip(S) \to M(n,S,\infty)$$
to
$$\SIso(C(K_n),\sigma_{C(K_n)})(\Qbf_p) \subset C(K_n)^{\vee} \otimes \Qbf_p = C(K_n)^{\vee} \otimes \Lip(S)$$
is a bijective map onto $\{(U,U^{-1}) \mid U \in \GL_n(\Z_p)\}$.
\end{proposition}

\begin{lemma}
\label{shortness of i_p-adic}
When $S = (\Q_p,|\cdot|_{\infty},1)$, $\iota_{S,\infty}$ is a short map.
\end{lemma}

\begin{proof}
Let $(f,s) \in {C(K_n)} \times S$. For any $(t_i)_{i=1}^{n} \in \Q_p^n$, we have
$$
\begin{array}{cll}
& & \|i_S(f,s)_0((t_i)_{i=1}^{n})\|_{S,{\infty}} = \left\|\left(\sum_{j=1}^{n} s f(e_{i,j,0}) t_j \right)_{i=1}^{n} \right\|_{S,{\infty}}\\
&=& \left\|\left(\left|\sum_{j=1}^{n} s f(e_{i,j,0}) t_j \right|_p \right)_{i=1}^{n} \right\|_{{\infty}} \\
& \leq & \left\|(\|(|s f(e_{i,j,0}) t_j|_p)_{j=1}^{n}\|_{{\infty}})_{i=1}^{n}\right\|_{{\infty}}\\
&=& \left\|(|s f(e_{i,j,0}) t_j|_p)_{i,j=1}^{n}\right\|_{{\infty}} =  \left\|(|s|_p |f|_{op} |e_{i,j,0}|_{C(K_n)} |t_j|_p)_{i,j=1}^{n}\right\|_{{\infty}} \\
& = & \left\|(|s|_p |f|_{op} |t_j|_p)_{i,j=1}^{n}\right\|_{{\infty}} = |s|_p |f|_{op} \left\|(|t_i|_p)_{i=1}^{n}\right\|_{{\infty}}\\
&=& |s|_p |f|_{C(K_n)^{\vee}} \|(t_i)_{i=1}^{n}\|_{S,{\infty}}
\end{array}
$$
by Proposition \ref{norm of e_i,j,b}, and hence $|i_S(f,s)_0|_{\End_S(S^n_{\ell^{\infty}})} \leq |s|_p |f|_{C(K_n)^{\vee}}$. Similarly, we have $|i_S(f,s)_1|_{\End_S(\uHom_S(S^n_{\ell^{\infty}},S))} \leq |s|_p |f|_{C(K_n)^{\vee}}$. Therefore $i_S$ gives a short $\Z$-bilinear map $C(K_n)^{\vee} \otimes^m S \to M(n,S,\infty)$. By Proposition \ref{universal-property-tensor-product}, $\iota_{S,\infty}$ is a short map.
\end{proof}

\begin{lemma}
\label{isometry of i_p-adic}
When $S = (\Q_p,|\cdot|_{\infty},1)$, $\iota_{S,\infty}$ is an isometric isomorphism.
\end{lemma}

\begin{proof}
We denote by $((\delta_{i,j,b})_{i,j=1}^{n})_{b=0}^{1}$ the $\Z$-linear basis of the underlying $\Z$-module of $C(K_n)^{\vee}$ dual to $((e_{i,j,b})_{i,j=1}^{n})_{b=0}^{1}$, which coincides with the image of the canonical basis $((E_{i,j,b})_{i,j=1}^{n})_{b=0}^{1}$ of $\End(V) \times \End(V^{\vee})$. For any $(i,j,b) \in \N^3$ with $1 \leq i,j \leq n$ and $b \leq 1$, we have $|\delta_{i,j,b}|_{C(K_n)^{\vee}} = |E_{i,j,b}|_{M(n,S,\infty)} = 1$ by Lemma \ref{embedding into the second dual of M(n,S,2)} and the fact from the classical $p$-adic analysis that the second dual embedding $M(n,S,\infty) \to \uHom_S(\uHom_S(M(n,S,\infty),S),S)$ is an isometry.

\vspace{0.1in}
Let $a \in C(K_n)^{\vee} \otimes \Qbf_p$. By Lemma \ref{shortness of i_p-adic}, we have
$$|a|_{C(K_n)^{\vee} \otimes \Qbf_p} \geq |\iota_{S,\infty}(a)|_{M(n,S,\infty)}.$$
By $a = \sum_{b=0}^{1} \sum_{i,j=1}^{n} a(e_{i,j,b}) \delta_{i,j,b}$, we obtain
$$
\begin{array}{cll}
|a|_{C(K_n)^{\vee} \otimes \Qbf_p} &=& \left|\sum_{b=0}^{1} \sum_{i,j=1}^{n} a(e_{i,j,b}) \delta_{i,j,b} \right|_{C(K_n)^{\vee} \otimes \Qbf_p}\\
&\leq& \|(\|(|a(e_{i,j,b}) \delta_{i,j,b}|_{C(K_n)^{\vee} \otimes \Qbf_p})_{i,j=1}^{n}\|_{\{1,\ldots,n\}^2,1})_{b=0}^{1}\|_{\{0,1\},1} \\
& = & \|(\|(|a(e_{i,j,b}) \delta_{i,j,b}|_{C(K_n)^{\vee} \otimes \Qbf_p})_{i,j=1}^{n}\|_{{\infty}})_{b=0}^{1}\|_{{\infty}} \\
& = & \|(\|(|a(e_{i,j,b})|_S |\delta_{i,j,b} \otimes 1|_{C(K_n)^{\vee} \otimes \Qbf_p})_{i,j=1}^{n}\|_{{\infty}})_{b=0}^{1}\|_{{\infty}}\\
&\leq& \|(\|(|a(e_{i,j,b})|_S |\delta_{i,j,b}|_{op} |1|_p)_{i,j=1}^{n}\|_{{\infty}})_{b=0}^{1}\|_{{\infty}} \\
& \leq & \|(\|(|a(e_{i,j,b})|_S)_{i,j=1}^{n}\|_{{\infty}})_{b=0}^{1}\|_{{\infty}}\\
&=& \left|\left(\sum_{i,j=1}^{n} a(e_{i,j,0}) E_{i,j,0},\sum_{i,j=1}^{n} a(e_{i,j.1}) E_{i,j,1} \right) \right|_{M(n,S,\infty)} \\
& = & |\iota_{S,\infty}(a)|_{M(n,S,\infty)}
\end{array}
$$
by Lemma \ref{presentation of tensor}. It implies $|a|_{C(K_n)^{\vee} \otimes \Qbf_p} = |\iota_{S,\infty}(a)|_{M(n,S,\infty)}$ Therefore $\iota_{S,\infty}$ is isometric. By
\begin{eqnarray*}
\dim_{\Q_p}(C(K_n)^{\vee} \otimes \Qbf_p) & = & \dim_{\Q_p}(\End(\Z^n)^{\vee \vee} \otimes_{\Z} \Q_p) = 2n^2 \\
& = & \dim_{\Q_p}(\End_{\Q_p}(\Q_p^n)^2) = \dim_{\Q_p}(M(n,S,\infty)),
\end{eqnarray*}
$\iota_{S,\infty}$ is bijective.
\end{proof}

\begin{lemma}
\label{closed unit ball of p-adic counterpart}
When $S =  (\Q_p,|\cdot|_{\infty},1)$, an $a \in C(K_n)^{\vee} \otimes \Qbf_p$ satisfies $|a(c)|_S \leq |c|_{C(K_n)^{\vee}}$ for any $c \in C(K_n)$ if and only if $|\iota_{S,\infty}(a)|_{M(n,S,\infty)} \leq 1$.
\end{lemma}

\begin{proof}
The opposite implication follows from the same argument as the proof of Lemma \ref{closed unit ball of real counterpart} except we apply Lemma \ref{isometry of i_p-adic} instead of Lemma \ref{isometry of i_real} and use $p_{\Qbf_p} = (1,1)$ instead of $p_{\Rbf} = (2,1)$. Suppose $|\iota_{S,\infty}(a)|_{M(n,S,\infty)} > 1$. Since the norm of $M(n,S,\infty)$ coincides with the $\ell^{\infty}$ norm of entries, there exists an $(i,j,b) \in \N^3$ with $1 \leq i,j \leq n$ and $b \leq 1$ such that the $(i,j,b)$-the entry of $\iota_{S,\infty}(a)|_{M(n,S,\infty)}$, which coincides with $a(e_{i,j,b})$ by the definition of $\iota_{S,\infty}$, is of $p$-adic norm $> 1$. Put $c := e_{i,j,b} \in C(K_n)$. We have $|a(c)|_p > 1 = |c|_{C(K_n)}$ by Proposition \ref{norm of e_i,j,b}.
\end{proof}

\begin{proof}[Proof of Proposition \ref{p-adic orthogonal group}]
By
\begin{eqnarray*}
\GL_n(\Z_p) = \{U \in \GL_n(\Q_p) \mid |U|_{\End_S(S^n_{\ell^{\infty}})} = |U^{-1}|_{\End_S(S^n_{\ell^{\infty}})} = 1\},
\end{eqnarray*}
the assertion immediately follows from Proposition \ref{coincidence of SIso}, Lemma \ref{isometry of i_p-adic}, and Lemma \ref{closed unit ball of p-adic counterpart}.
\end{proof}

%************************************************************************
\subsection{Data availability statement}
Not applicable.

%************************************************************************
\subsection{Conflict of interest}
On behalf of all authors, the corresponding author states that there is no conflict of interest.

\newpage

% *************************
% Appel de la bibliographie
% *************************
\bibliographystyle{alpha}
%\bibliography{/share/nfs/users/imj-ao/fpaugam/travail/fred}
%\bibliography{/home/visitor/fpaugam/travail/fred}
\bibliography{fred}

\def\cprime{$'$} \def\cprime{$'$} \def\cprime{$'$} \def\cprime{$'$}
  \def\cprime{$'$} \def\cprime{$'$}
\begin{thebibliography}{KMRT98}

\bibitem[Art67]{E-Artin1}
Emil Artin.
\newblock {\em Algebraic numbers and algebraic functions}.
\newblock Gordon and Breach Science Publishers, New York, 1967.

\bibitem[Ber90]{Berkovich1}
Vladimir~G. Berkovich.
\newblock {\em Spectral theory and analytic geometry over non-{A}rchimedean
  fields}, volume~33 of {\em Mathematical Surveys and Monographs}.
\newblock American Mathematical Society, Providence, RI, 1990.

\bibitem[KMRT98]{Involutions}
Max-Albert Knus, Alexander Merkurjev, Markus Rost, and Jean-Pierre Tignol.
\newblock {\em The book of involutions}.
\newblock American Mathematical Society, Providence, RI, 1998.
\newblock With a preface in French by J.\ Tits.

\bibitem[{Pau}14]{Fred-overconvergent-global-analytic-geometry}
F.~{Paugam}.
\newblock {Overconvergent global analytic geometry}.
\newblock {\em ArXiv e-prints}, October 2014.

\bibitem[Poi07]{Poineau}
J\'er\^ome Poineau.
\newblock {\em Espaces de Berkovich sur {$\bold Z$}}.
\newblock Th\`ese. universit\'e de Rennes 1, 2007.

\bibitem[Tat67]{Tate-these}
J.~T. Tate.
\newblock Fourier analysis in number fields, and {H}ecke's zeta-functions.
\newblock In {\em Algebraic Number Theory (Proc. Instructional Conf., Brighton,
  1965)}, pages 305--347. Thompson, Washington, D.C., 1967.

\end{thebibliography}

\end{document}